\documentclass[12pt, reqno]{amsart}
\usepackage{amsmath, amsthm, amscd, amsfonts, amssymb, graphicx, color}
\usepackage[bookmarksnumbered, colorlinks, plainpages]{hyperref}
\allowdisplaybreaks

\usepackage{microtype} 

\hypersetup{
    colorlinks=true,    
    linkcolor=black,    
    filecolor=black,
    urlcolor=blue,      
    citecolor=black,    
}

\usepackage{titletoc}

\titlecontents{section}
  [0em]                              
  {\vspace{0.5em}}                   
  {\thecontentslabel\hspace{0.8em}}  
  {}                                 
  {\titlerule*[0.5pc]{.}\contentspage} 

\titlecontents{subsection}
  [1.5em]
  {}
  {\thecontentslabel\hspace{0.8em}}  
  {}
  {\titlerule*[0.5pc]{.}\contentspage}

\newtheorem{theorem}{Theorem}[section]
\newtheorem{lemma}[theorem]{Lemma}
\newtheorem{proposition}[theorem]{Proposition}
\newtheorem{corollary}[theorem]{Corollary}
\theoremstyle{definition}
\newtheorem{definition}[theorem]{Definition}
\newtheorem{example}[theorem]{Example}

\theoremstyle{remark}
\newtheorem{remark}[theorem]{Remark}
\numberwithin{equation}{section}
\usepackage{titlesec}
\titleformat{\section}[hang]{\normalfont\Large\bfseries}{\thesection.}{0.5em}{}
\titleformat{\subsection}[hang]{\normalfont\large\bfseries}{\thesubsection.}{0.5em}{}
\titleformat{\subsubsection}[hang]
  {\normalfont\normalsize\bfseries}
  {\thesubsubsection.}
  {0.5em}
  {}
\DeclareMathOperator{\Irr}{Irr}
\DeclareMathOperator{\Pol}{Pol}

\begin{document}

\setcounter{page}{1}

\title[GH convergence of spectral truncation for quantum groups]{Quantum Gromov--Hausdorff Convergence of Spectral Truncations for Compact Quantum Groups with Proper Length Functions}

\author{Xintao Peng, Qin Wang}

\begin{abstract}
We study the quantum Gromov--Hausdorff convergence of spectral truncations of \(C(G)\), where \(G\) is a coamenable compact quantum group equipped with a proper length function. Using higher-order commutator seminorms, we establish sufficient conditions for the induced Lipschitz seminorms to define compact quantum metric structures. Furthermore, for the Dirac operators associated with the length function, we prove that their spectral truncations converge to the original space in the quantum Gromov--Hausdorff distance. Our results apply to \(SU(N)\), \(SO(N)\), and related examples, including coamenable compact quantum groups of Kac type with polynomial growth.
\end{abstract}

\maketitle


\section{Introduction}
Rieffel \cite{9,32} introduced the theory of compact quantum metric spaces, which extends the Gromov--Hausdorff distance to the noncommutative setting. Along this line, spectral truncations in noncommutative geometry have been extensively studied. For instance, Connes and van Suijlekom \cite{1}, van Suijlekom \cite{5}, and Leimbach and van Suijlekom \cite{4} established convergence results in the torus setting. Estrada and van Suijlekom \cite{7} extended these convergence results to compact metric groups. Moving towards quantum groups, Rieffel \cite{11} investigated Fourier truncations for compact quantum groups and finitely generated groups, while Leimbach \cite{6} studied Peter--Weyl truncations and developed a general framework for proving convergence under suitable assumptions. More recently, the study of spectral truncations are also relevant to the theory of spectral localizers; see \cite{Verhoeven} for details.
\\

An important step toward such a theory was taken by Austad and Kyed~\cite{16}. Let \(G\) be a compact quantum group equipped with a proper length function \(\ell:\Irr(G)\to[0,\infty)\). As shown in~\cite{23}, this yields a spectral triple
\[
\bigl(\Pol(G),\, L^{2}(G),\, D_{\ell}\bigr),
\]
where \(\Pol(G)\) is the Hopf \(*\)-algebra of representative functions, \(L^{2}(G)\) is the GNS Hilbert space of the Haar state, and \(D_{\ell}\) acts on matrix coefficients by
\[
D_{\ell}\bigl(\Lambda(u_{i,j}^{\alpha})\bigr)=\ell(\alpha)\,\Lambda(u_{i,j}^{\alpha}).
\]
However, having this spectral triple is not sufficient to define a compact quantum metric space; one must additionally verify that the associated Lipschitz seminorm induces the weak\(^{*}\) topology on the state space.

Austad and Kyed \cite{16} established that for coamenable compact quantum groups of Kac type, the full algebra and its central subalgebra share the same compact quantum metric properties. Specifically, if \(L_{\ell}\) denotes the Lipschitz norm associated with \(\ell\) and \(L_{z,\ell}\) its restriction to the center, then \((C(G),L_{\ell})\) is a compact quantum metric space if and only if \((C_{z}(G),L_{z,\ell})\) is. This equivalence means the metric structure can be analyzed entirely at the level of central functions, or equivalently, the fusion algebra
\[
F(G)=\operatorname{span}_{\mathbb{C}}\{\alpha:\alpha\in\Irr(G)\}.
\]
Thus, the metric information on \(C(G)\) is fully captured by either its central subalgebra \(C_{z}(G)\) or the fusion algebra \(F(G)\).

Bhowmick, Voigt, and Zacharias~\cite{23} initiated the study of compact quantum metric spaces arising from discrete quantum groups with the rapid decay property. Motivated by their approach, we explore how higher-order commutators generate Lipschitz norms on fusion algebras, leading to compact quantum metric structures. In particular, we show that certain coamenable compact quantum groups of Kac type admit such structures. By applying the method of Li~\cite{8}, we construct a bi-invariant Lipschitz norm on the truncation space. Our investigation shares the general philosophy of Leimbach's framework~\cite{6}, but our primary focus shifts to spectral truncations equipped with higher-order Lip-seminorms of the form \(L_\ell^k\). We establish sufficient conditions for these seminorms to induce compact quantum metric spaces and apply the resulting criteria to concrete classes of compact quantum groups. Building on these conditions, we prove that the corresponding spectral truncations converge to the original quantum metric space in the quantum Gromov--Hausdorff distance. While Leimbach~\cite{6} provided a general convergence framework for Peter--Weyl truncations, our work concentrates on the specific metric properties required for spectral truncation convergence across a range of compact quantum groups.

We begin by considering the spectral truncations~\cite{4} associated with a spectral triple \((A,\mathcal{H},D)\), where \(A\) acts on a Hilbert space \(\mathcal{H}\) and \(D\) is a self-adjoint operator with compact resolvent. Given a cutoff \(\Lambda\), let \(P_{\Lambda}\) denote the spectral projection of \(D\) onto the eigenspaces corresponding to eigenvalues of modulus at most \(\Lambda\). This yields the truncated operator system \(P_{\Lambda}AP_{\Lambda}\) acting on \(P_{\Lambda}\mathcal{H}\), equipped with the truncated operator \(D_{\Lambda}=D|_{P_{\Lambda}\mathcal{H}}\). The central question is to determine the conditions under which these truncations converge to the original algebra in the quantum Gromov--Hausdorff distance.

For compact quantum groups, our primary object is the fusion algebra
\[
F(G)=\operatorname{span}_{\mathbb{C}}\{\alpha:\alpha\in\Irr(G)\},
\]
endowed with the natural \(\ell^{2}\)-structure inherited from the representation dimensions \(d(\alpha)\). The fusion algebra \(F(G)\) is said to have the rapid decay property if there exist constants \(C>0\) and \(s>0\) such that
\[
\|\Lambda(f)\|\le C\|f\|_{2,s}
\]
for all \(f\in F(G)\). In addition, we introduce the definitions of polynomial growth and strong polynomial growth relative to a length function \(\ell\). As our first main result, we prove that polynomial growth guarantees that the higher-order commutator seminorms induce compact quantum metric structures.

\begin{theorem}
Let \(G\) be a compact quantum group equipped with a proper length function \(\ell\). Suppose that \(G\) has polynomial growth of order \(s\geq 0\). Then, for every positive integer \(k>s\), the seminorm \(L_\ell^k(f)\) defines a compact quantum metric space structure on the operator-norm closure of \(F(G)\).
\end{theorem}

This framework naturally includes connected compact Lie groups such as \(SU(N)\) and \(SO(N)\) equipped with suitable Lipschitz norms. Furthermore, for coamenable compact quantum groups of Kac type, we establish a direct link between the metric structures on the full algebra and its fusion algebra.

\begin{theorem}
If \(G\) is a coamenable compact quantum group of Kac type and \((F(G),L_\ell^k)\) is a compact quantum metric space, then \((C(G),L_\ell^k)\) is a compact quantum metric space.
\end{theorem}

Having established these metric structures, we turn our attention to the spectral truncations of \(C(G)\). Let \(C(G)_\Lambda=P_{\Lambda}C(G)P_{\Lambda}\) be the truncated operator system, and let \(L_{\ell,\Lambda}^{k,\Delta_\Lambda}\) denote the seminorms induced by the coactions. Our strategy builds upon the construction of Li~\cite{8}, which employs ergodic actions of compact quantum groups to induce Lip-norms on truncated operator systems. While drawing on ideas from Leimbach~\cite{6} regarding the convergence of Peter--Weyl truncations, we shift our focus to spectral truncations equipped with higher-order Lip-seminorms. Under suitable assumptions, these tools yield the following convergence result.

\begin{theorem}\label{1}
Let \(G\) be a coamenable compact quantum group equipped with a proper length function \(\ell\). Suppose that \(\bigl(C(G),L_{\ell}^{k}\bigr)\) is a compact quantum metric space. Then
\[
\operatorname{dist}_{q}\left(
\bigl(C(G),L_{\ell}^{k}\bigr),
\bigl(P_{\Lambda}C(G)P_{\Lambda},
L_{\ell,\Lambda}^{k,\Delta_\Lambda}\bigr)
\right)
\longrightarrow 0
\qquad\text{as }\Lambda\to\infty.
\]
\end{theorem}

The proof proceeds by explicitly constructing suitable maps \(\tau_\Lambda:C(G)\to C(G)_{\Lambda}\) and \(\rho^{\phi}:C(G)_{\Lambda}\to C(G)\), and subsequently showing that they satisfy the hypotheses of Definition \ref{definition Ck}.
\\

The paper is organized as follows. Section~2 recalls the necessary background on
quantum Gromov--Hausdorff convergence, compact quantum groups, and fusion
algebras. Section~3 investigates compact quantum metric spaces arising from
fusion algebras with the rapid decay property. In addition, for coamenable Kac-type quantum groups $G$, we prove that
$\bigl(C(G),L_{\ell}^{k}\bigr)$ is a compact quantum metric space if and
only if
$\bigl(C_{z}(G),L_{z,\ell}^{k}\bigr)$ is a compact quantum metric space. Section~4 establishes quantum
Gromov--Hausdorff convergence results for spectral truncations of
$C(G)$ under the assumption that
$\bigl(C(G),L_{\ell}^{k}\bigr)$ is a compact quantum metric space and
 presents several examples illustrating the applicability of our results.

\section{Preliminaries}

\subsection{ Quantum Gromov-Hausdorff Distance }
This subsection reviews the concept of compact quantum metric spaces 
and the quantum analogue of Gromov--Hausdorff convergence introduced 
by Marc A.~Rieffel in~\cite{9,32,33}. The relevant definitions and key 
properties are recalled below.

\begin{definition}\label{Definition2.1}
Let $A$ be a complete operator system. A seminorm 
$L \colon A \to [0,\infty]$ is called a \emph{Lipschitz seminorm} 
if $L$ satisfies the following conditions:
\begin{itemize}
    \item[(1)] $L$ is \emph{densely defined}, meaning that the domain 
    $\operatorname{Dom}(L) := \{a \in A : L(a) < \infty\}$ is a 
    norm-dense subspace of $A$;
    
    \item[(2)] $L(a^*) = L(a)$ for all $a \in A$, and the kernel of $L$ 
    is exactly the scalar multiples of the identity, i.e.\ 
    $L(a) = 0 \iff a \in \mathbb{C}\cdot e_A$.
\end{itemize}
\end{definition}

\begin{definition}\label{Definition2.2}
Let $A$ be an operator system equipped with a Lipschitz seminorm 
$L \colon A \to [0,+\infty]$. We define a metric $d^L$ on the state 
space $\mathcal{S}(A)$ by
\begin{equation}\label{eq:dist}\tag{1}
    d^L(\phi, \psi) 
    := \sup\bigl\{ |\phi(a) - \psi(a)| : a \in A,\ L(a) \leq 1 \bigr\},
    \qquad \phi, \psi \in \mathcal{S}(A).
\end{equation}
We require that the topology induced by $d^L$ on $\mathcal{S}(A)$ coincides with 
the weak$^{*}$ topology. The pair $(A, L)$ is then called a 
\emph{compact quantum metric space}. In this case, $L$ is referred to 
as a \emph{Lipschitz norm}.
\end{definition}

\begin{remark}
The metric $d^L$ defined in (1) is referred to as the 
Monge--Kantorovi\v{c} metric. The condition that $d^L$ metrises 
the weak$^{*}$ topology ensures that $\mathcal{S}(A)$ is compact, hence the 
terminology.
\end{remark}

Next, we recall the \emph{quantum Gromov--Hausdorff distance} 
between two compact quantum metric spaces.

\begin{definition}(\cite{9})\label{Definition2.4}
Let $(A, L_A)$ and $(B, L_B)$ be compact quantum metric spaces. 
Let $\mathcal{M}(L_A, L_B)$ be the set of Lipschitz norms on $A \oplus B$ 
which induce $L_A$ and $L_B$. The \emph{quantum Gromov--Hausdorff 
distance} between them, denoted $\operatorname{dist}_q(A, B)$, is 
defined by
\[
\operatorname{dist}_q(A, B) = \inf\bigl\{ 
\operatorname{dist}_H^{d^L}(\mathcal{S}(A), \mathcal{S}(B)) : L \in \mathcal{M}(L_A, L_B) 
\bigr\}.
\]
\end{definition}

We now recall a proposition that provide an upper bound for the 
quantum Gromov--Hausdorff distance.

\begin{definition}(\cite{9})
    Let $(A, L_A)$ be a compact quantum metric space with state space $\mathcal{S}(A)$. The \emph{diameter} is given by
    \[
        \operatorname{diam}(A, L_A) 
        := \sup \Bigl\{ d^{L_A}(\mu, \nu) \;\big|\; \mu, \nu \in \mathcal{S}(A) \Bigr\}.
    \]
\end{definition}

The following approximation criterion is useful for spectral
truncations.We state the following proposition in the form of \cite{28}, Proposition 2.2.4. It gives a sufficient condition for estimating the complete Gromov–Hausdorff distance, analogous to \cite{5}, Proposition 4, and relies on methods used in \cite{32},\cite{33} and formalized via bridges in \cite{9}.

\begin{proposition}[\cite{28} Proposition 2.2.4]\label{prop:gromov hausdorff converges}
 
Let $(A, L_A)$ and $(B, L_B)$ be compact quantum metric spaces and suppose that $\Phi: A \to B$ and $\Psi: B \to A$ are two unital positive maps satisfying that
\begin{enumerate}
    \item[(1)] there exist $C, C' > 0$ such that
    \[
    L_B(\Phi(a)) \leqslant C \cdot L_A(a) \quad \text{and} \quad L_A(\Psi(b)) \leqslant C' \cdot L_B(b)
    \]
    for all $a \in A$ and $b \in B$;
    \item[(2)] there exist $\varepsilon, \varepsilon' > 0$ such that
    \[
    \|\Psi\Phi(a) - a\| \leqslant \varepsilon \cdot L_A(a) \quad \text{and} \quad \|\Phi\Psi(b) - b\| \leqslant \varepsilon' \cdot L_B(b)
    \]
    for all $a \in A$ and $b \in B$.
\end{enumerate}
Then the quantum Gromov--Hausdorff distance $\operatorname{dist}_{q}(A,B)$ is dominated by
\[
\max\bigl\{ \operatorname{diam}(A, L_A) \cdot |1 - 1/C| + \varepsilon/C, \operatorname{diam}(B, L_B) \cdot |1 - 1/C'| + \varepsilon'/C' \bigr\}.
\]
\end{proposition}

Now we focus on spectral triples and the compact quantum metric spaces arising from them. 
We recall the following notion, which originates from Leimbach and van Suijlekom~\cite{4}.

\begin{definition}(\cite{4})\label{def:op-sys-spectral-triple}
An \emph{operator system spectral triple} is a triple
$(\mathcal{E}, \mathcal{H}, D)$ where $\mathcal{E}$ is a dense subspace
of a (concrete) operator system $E$ in $\mathcal{B}(\mathcal{H})$,
$\mathcal{H}$ is a Hilbert space, and $D$ is a self‑adjoint operator
on $\mathcal{H}$ with compact resolvent such that $[D, T]$ is a bounded
operator for all $T \in \mathcal{E}$.
\end{definition}

Building on this framework from~\cite{4}, we now consider sequences of operator system spectral triples and their limits under spectral truncation.
The limit structure is captured by the following notion of an approximate order isomorphism.

Let $(\mathcal{E}, \mathcal{H}, D)$ be an operator system spectral triple as in Definition~\ref{def:op-sys-spectral-triple}.
Following the standard spectral‑truncation construction, we define the truncated objects as follows.
A spectral truncation is given by a spectral projection $P_\Lambda$ of $D$ onto the eigenspaces with eigenvalues of modulus $\le \Lambda$.
The space $\mathcal{E}$ is then truncated to
$\mathcal{E}_\Lambda := P_\Lambda \mathcal{E} P_\Lambda$, acting on
$\mathcal{H}_\Lambda := P_\Lambda \mathcal{H}$, and $D$ restricts to the
self‑adjoint operator $D_\Lambda := P_\Lambda D P_\Lambda$ on
$P_\Lambda \mathcal{H}$.

We define the associated Lipschitz seminorms: for $k \in \mathbb{N}$, 
$a \in \mathcal{E}$ and $b \in \mathcal{E}_\Lambda$, set
\begin{align*}
   & L^k(a)   := \|\delta^k(a)\|   = \bigl\|[D, [D, \dots [D, a] \dots ]]\bigr\|,
\\
&L^k_\Lambda(b) := \|\delta^k_\Lambda(b)\| 
                = \bigl\|[D_\Lambda, [D_\Lambda, \dots [D_\Lambda, b] \dots ]]\bigr\|.
\end{align*}

where the commutator is taken $k$ times. 
If these Lipschitz seminorms are actually Lipschitz norms , then we can consider the quantum 
Gromov--Hausdorff distance between the corresponding compact quantum 
metric spaces. In~\cite{5}, van Suijlekom constructs a $C^1$-approximate order isomorphism between an operator system spectral triple and its truncated spectral triple to bound their quantum Gromov--Hausdorff distance. Following van Suijlekom, we use the following k-th order variant.

\begin{definition}\label{definition Ck}
Let $\{(\mathcal{E}_\Lambda, \mathcal{H}_\Lambda, D_\Lambda)\}_\Lambda$ be a sequence of operator system spectral triples, and let $(\mathcal{E}, \mathcal{H}, D)$ be a fixed operator system spectral triple. 
Let $L_\Lambda$ denote the Lipschitz seminorm on $\mathcal{E}_\Lambda$.

A pair of linear maps 
\[
\rho_\Lambda\colon \mathcal{E}_\Lambda\to \mathcal{E},\qquad \tau_\Lambda\colon \mathcal{E}\to \mathcal{E}_\Lambda
\]
is said to form a \emph{$C^k$-approximate order isomorphism} if the following conditions hold:
\begin{enumerate}
    \item Both $\rho_\Lambda$ and $\tau_\Lambda$ are unital and positive.

    \item Both $\rho_\Lambda$ and $\tau_\Lambda$ are $C^k$-contractive, where $k \in \mathbb{N}$. That is, they are contractive with respect to the operator norm $\|\cdot\|$ as well as the corresponding Lipschitz seminorms: $\tau_\Lambda$ is contractive with respect to $L^k$, and $\rho_\Lambda$ is contractive with respect to $L_\Lambda$.

    \item There exists a sequence $\varepsilon_\Lambda \to 0$ such that
    \[
    \|a - \rho_\Lambda\circ\tau_\Lambda(a)\| \leq \varepsilon_\Lambda L^k(a),\qquad \forall a\in\mathcal{E},
    \]
    \[
    \|b - \tau_\Lambda\circ\rho_\Lambda(b)\| \leq \varepsilon_\Lambda L_\Lambda(b),\qquad \forall b\in\mathcal{E}_\Lambda.
    \]
\end{enumerate}
We call $(\rho_\Lambda,\tau_\Lambda)$ a \emph{$C^k$-approximate order isomorphism} between the spectral triples.
\end{definition}

To study the quantum Gromov--Hausdorff distance, we observe that the Lipschitz seminorm $L_\Lambda$ is not necessarily equivalent to $L_\Lambda^k$. In particular, $L_\Lambda^k$ may fail to satisfy Conditions 2 and 3 in the preceding definition. As shown in Theorem~5 of~\cite{5}, a \emph{$C^1$-approximate order isomorphism} is closely related to quantum Gromov--Hausdorff convergence. We now generalize this setting and investigate the relationship between \emph{$C^k$-approximate order isomorphism} and quantum Gromov--Hausdorff convergence.

\begin{theorem}\label{theorem:C^k}
Let $(\rho_\Lambda,\tau_\Lambda)$ be a $C^k$-approximate order isomorphism between the operator system spectral triples $(\mathcal{E}_\Lambda, \mathcal{H}_\Lambda, D_\Lambda)$ and $(\mathcal{E}, \mathcal{H}, D)$. If $(E,L_\ell^k)$ and $(E_\Lambda,L_\Lambda)$ are compact quantum metric spaces, then the corresponding quantum Gromov--Hausdorff distance satisfies
\[
\operatorname{dist}_{\mathrm{q}}(E,E_\Lambda)\leq \varepsilon_\Lambda.
\]
\end{theorem}

\begin{proof}
Since the maps $\rho_\Lambda$ and $\tau_\Lambda$ satisfy all the assumptions in Proposition \ref{prop:gromov hausdorff converges}, the desired estimate for the quantum Gromov--Hausdorff distance immediately follows, yielding
\[
\operatorname{dist}_{\mathrm{q}}(E,E_\Lambda)\leq \varepsilon_\Lambda.
\]
\end{proof}

\subsection{Compact Quantum Groups}

Quantum groups emerged in the 1980s as a framework for studying symmetries 
of noncommutative spaces and exactly solvable models in statistical mechanics 
and quantum field theory. Among them, compact quantum groups, introduced by 
Woronowicz~\cite{12} in the late 1980s, provide a natural generalization of 
compact topological groups to the noncommutative setting. They play a central 
role in noncommutative geometry, representation theory, and the theory of 
operator algebras. In this subsection we recall the basic definition and some 
well‑known facts that will be needed later.For further details, we refer the reader to~\cite{17,26}.

\begin{definition}\label{def:compact-quantum-group}
A \emph{compact quantum group} is a pair $(\mathcal{A}, \Delta)$ where 
$\mathcal{A}$ is a unital $C^*$-algebra and 
$\Delta : \mathcal{A} \to \mathcal{A} \otimes \mathcal{A}$ is a unital 
$*$-homomorphism, called the \emph{comultiplication}, satisfying
\begin{enumerate}
    \item $(\Delta \otimes \operatorname{id}_{\mathcal{A}}) \Delta 
          = (\operatorname{id}_{\mathcal{A}} \otimes \Delta) \Delta$ 
          as homomorphisms $\mathcal{A} \to \mathcal{A} \otimes \mathcal{A} 
          \otimes \mathcal{A}$ (coassociativity);
    \item the spaces 
          $(\mathcal{A} \otimes 1_{\mathcal{A}}) \Delta(\mathcal{A}) 
          := \operatorname{span}\{(a \otimes 1_{\mathcal{A}}) \Delta(b) : 
          a, b \in \mathcal{A}\}$ and 
          $(1_{\mathcal{A}} \otimes \mathcal{A}) \Delta(\mathcal{A})$ 
          are dense in $\mathcal{A} \otimes \mathcal{A}$ 
          (the cancellation property).
\end{enumerate}
Here the tensor product is the minimal $C^*$-tensor product.
\end{definition}

A fundamental fact about compact quantum groups is the existence and 
uniqueness of a \emph{Haar state} $h : \mathcal{A} \to \mathbb{C}$. 
It is a state satisfying the left and right invariance conditions
\[
(\operatorname{id}_{\mathcal{A}} \otimes h) \circ \Delta 
= h(\cdot) 1_{\mathcal{A}} 
= (h \otimes \operatorname{id}_{\mathcal{A}}) \circ \Delta.
\]
In general, the Haar state may not be a trace, and it may not be faithful 
on $\mathcal{A}$.

Two important classes of compact quantum groups are distinguished by 
additional properties of the Haar state. A compact quantum group 
$(\mathcal{A}, \Delta)$ is said to be of \emph{Kac type} if its Haar state 
is a trace, i.e.\ $h(ab) = h(ba)$ for all $a,b \in \mathcal{A}$. 
Equivalently, there exists a bounded antipode $S$ on a dense 
$*$-subalgebra satisfying $S^2 = \operatorname{id}$.

A compact quantum group is called \emph{coamenable} if its dual discrete 
quantum group is amenable. For a discrete group $\Gamma$, the compact 
quantum group $C^*(\Gamma)$ is coamenable precisely when $\Gamma$ is an 
amenable group.

\begin{example}[Discrete group duals]
Let \(\Gamma\) be a discrete group. Denote by \(C^*(\Gamma)\) the full group $C^*$-algebra. There is a natural comultiplication \(\Delta : C^*(\Gamma) \to C^*(\Gamma) \otimes C^*(\Gamma)\) determined by \(\Delta(\delta_g) = \delta_g \otimes \delta_g\) for the canonical unitaries \(\delta_g\) (\(g \in \Gamma\)). Together with the usual involution and unit, this makes \((C^*(\Gamma), \Delta)\) a compact quantum group. Its Haar state is given by the canonical trace \(h(\lambda_g) = \delta_{g,e}\) for all \(g\).  It can be regarded as the dual of the discrete group \(\Gamma\).
\end{example}

For the purpose of treating situations analogous to those in~\cite{16}, involving higher‑dimensional commutators, we introduce the following definition.
Let $(\mathcal{H},\Lambda)$ denote the GNS representation of $C(G)$ associated with the Haar state, so that $\mathcal{H}=L^{2}(G)$ and $\Lambda\colon C(G)\to L^{2}(G)$ is the canonical GNS map. We adopt the compact quantum group conventions from~\cite{12,16,26}.

On $L^{2}(G)\hat{\otimes}L^{2}(G)$, the \emph{left fundamental unitary} $W$ and \emph{right fundamental unitary} $V$ are defined on $\Lambda\otimes\Lambda\bigl(C(G)\odot C(G)\bigr)$ by
\begin{align*}
W^{*}\bigl(\Lambda(x)\otimes\Lambda(y)\bigr)
&= \Lambda\otimes\Lambda\bigl(\Delta(y)(x\otimes 1)\bigr), \\
V\bigl(\Lambda(x)\otimes\Lambda(y)\bigr)
&= \Lambda\otimes\Lambda\bigl(\Delta(x)(1\otimes y)\bigr),
\end{align*}
for all $x,y\in C(G)$. They implement the comultiplication of $C(G)$:
\begin{equation}
\Delta(x) = W^{*}(1\otimes x)W = V(x\otimes 1)V^{*},\qquad x\in C(G). \tag{2.3}
\end{equation}
$\Delta$ extends uniquely to a normal $*$-homomorphism
\[
\Delta\colon L^{\infty}(G) \longrightarrow L^{\infty}(G)\,\hat{\otimes}\,L^{\infty}(G),
\]
endowing $L^{\infty}(G)$ with the structure of a compact von Neumann‑algebraic quantum group~\cite{12,26}.

The \emph{flip unitary} $\Sigma\in\mathcal{B}\bigl(L^{2}(G)\hat{\otimes}L^{2}(G)\bigr)$ is given by
\[
\Sigma\bigl(\Lambda(x)\otimes\Lambda(y)\bigr)=\Lambda(y)\otimes\Lambda(x),\qquad x,y\in C(G),
\]
which implements the flip automorphism $\sigma(a\otimes b)=\Sigma(a\otimes b)\Sigma = b\otimes a$ for $a,b\in\mathcal{B}(L^{2}(G))$.

We adopt the following notation:
\begin{enumerate}
\item \emph{Leg‑numbering:} For $T\in\mathcal{B}\bigl(L^{2}(G)\hat{\otimes}L^{2}(G)\bigr)$,
\[
T_{12}=T\otimes 1,\quad T_{13}=(1\otimes\sigma)(T\otimes 1),\quad T_{23}=1\otimes T,
\]
with analogous definitions for longer tensor products.
\item \emph{Sweedler notation:} For $a\in C(G)$ with $\Delta(a)\in C(G)\odot C(G)$, write $\Delta(a)=a_{(0)}\otimes a_{(1)}$, with implicit summation. The same convention applies for coactions of $G$.
\end{enumerate}
Since $W,\Sigma V\Sigma$ are unitary corepresentations and $W_{23}$ commutes with $(\Sigma V\Sigma)_{13}$, $Z=W\Sigma V\Sigma\in L^{\infty}(G)\bar{\otimes}\mathcal{B}(L^{2}(G))$ is also a unitary corepresentation.

A \emph{unitary corepresentation} of $G$ on a Hilbert space $H$ is a unitary element $U\in L^{\infty}(G)\bar{\otimes}\mathcal{B}(H)$ satisfying
\begin{equation}
(\Delta\otimes\mathrm{id})(U)=U_{13}U_{23}. \tag{2.4}
\end{equation}
The fundamental unitaries $W$ and $V$ obey the pentagon identity; for $W$, this reads $W_{12}W_{13}W_{23}=W_{23}W_{12}$. Combined with condition (2.4), this identity ensures that both $W$ and $\Sigma V\Sigma$ are unitary corepresentations.

Finite‑dimensional unitary corepresentations decompose into irreducible ones. The unit $1\in C(G)$ yields the trivial one‑dimensional irreducible corepresentation with equivalence class $e$.
For each $\alpha\in\operatorname{Irr}(G)$, let $u^{\alpha}\in L^{\infty}(G)\bar{\otimes}\mathcal{B}(H_{\alpha})$ be the irreducible corepresentation on the finite‑dimensional space $H_{\alpha}$, $d(\alpha)=\dim H_{\alpha}$. Under the identification $L^{\infty}(G)\bar{\otimes}\mathcal{B}(H_{\alpha})\cong M_{d(\alpha)}(L^{\infty}(G))$, its matrix coefficients $u_{ij}^{\alpha}$ lie in $C(G)$. The span
\[
\Pol(G):=\operatorname{span}_{\mathbb{C}}\bigl\{u_{ij}^{\alpha}\mid \alpha\in\operatorname{Irr}(G),\,1\le i,j\le d_{\alpha}\bigr\}
\]
is a dense Hopf $*$-subalgebra of $C(G)$, with comultiplication
\begin{equation}
\Delta(u_{ij}^{\alpha})=\sum_{k=1}^{d_{\alpha}} u_{ik}^{\alpha}\otimes u_{kj}^{\alpha}. \tag{2.5}
\end{equation}
$\Pol(G)$ is equipped with an antipode $S$ and a counit $\varepsilon$, see~\cite[Theorem~5.4.1]{26}. The operators $W^{*},V$ (as well as $W,V^{*}$) bijectively map $\Pol(G)\odot\Pol(G)$ onto itself~\cite[Theorem~1.3.18]{26}.

\subsection{Fusion algebras}
In~\cite{14}, Hiai and Izumi study fusion algebras, a structure intimately connected with the representation theory of quantum groups.
Fusion algebras, also called based rings, furnish an algebraic framework unifying the representation theory of compact groups and the combinatorial structure of discrete groups. These algebras first emerged systematically in 
the study of conformal field theory and have since found profound applications 
in subfactor theory, quantum groups, and tensor categories. We briefly recall their setup below.

\begin{definition}[\cite{14} Definition 1.1 ]\label{def:fusion_algebra}
An algebra $R = \bigoplus_{\xi \in I} \mathbb{C}\xi$ over $\mathbb{C}$ with a 
basis $I$ and an identity $e \in I$ is called a \emph{fusion algebra} if it 
satisfies the following axioms:

\begin{enumerate}
    \item[(i)] \textbf{Structure constants.} There exist non-negative integers 
    $N_{\xi,\eta}^{\zeta}$ for $\xi, \eta, \zeta \in I$ such that
    \[
    \xi\eta = \sum_{\zeta \in I} N_{\xi,\eta}^{\zeta}\,\zeta.
    \]
    These integers $N_{\xi,\eta}^{\zeta}$ are called the 
    \emph{structural constants} of $R$.
    
    \item[(ii)] \textbf{Conjugation.} There exists an involutive map 
    $\xi \mapsto \bar{\xi}$ on $I$, called \emph{conjugation}, whose 
    $\mathbb{Z}$-linear extension to $R$ satisfies
    \[
    \overline{\xi\eta} = \bar{\eta}\bar{\xi}, \quad \xi, \eta \in I.
    \]
    
    \item[(iii)] \textbf{Frobenius reciprocity.} The structural constants obey 
    the symmetry relations
    \[
    N_{\xi,\eta}^{\zeta} = N_{\bar{\xi},\zeta}^{\eta} 
    =N_{\zeta,\bar{\eta}}^{\xi}, \quad \xi, \eta, \zeta \in I.
    \]
    
    \item[(iv)] \textbf{Dimension function.} There exists a function 
    $d: I \to [1,\infty)$, called the \emph{dimension function} (or 
    \emph{Frobenius--Perron dimension}), such that 
    $d(\xi) = d(\bar{\xi})$ for all $\xi \in I$ and
    \[
    d(\xi)d(\eta) = \sum_{\zeta \in I} N_{\xi,\eta}^{\zeta}\,d(\zeta), 
    \quad \xi, \eta \in I.
    \]
\end{enumerate}
\end{definition}

\begin{remark}
The dimension function $d$ is uniquely determined by the above conditions 
and satisfies $d(e) = 1$. Moreover, the Frobenius reciprocity (iii) ensures 
that the algebra $R$ is semisimple when equipped with the bilinear form 
$\langle \xi, \eta \rangle = \delta_{\xi, \bar{\eta}}$.
\end{remark}

We now present two classical examples of fusion algebras: group algebras of 
discrete groups and representation rings of compact quantum groups.

\begin{example}\label{ex:group-algebra}
Let $G$ be a discrete group and consider $R = \mathbb{Z}[G] = \bigoplus_{g \in G} \mathbb{Z}g$, which forms a fusion algebra 
with the following data:
\begin{itemize}
    \item Basis: $\Xi_0 = G$ with identity $e$ being the group identity;
    \item Multiplication: $N_{g_1,g_2}^{h} = \delta_{h,g_1g_2}$ for 
    $g_1, g_2, h \in G$;
    \item Conjugation: $\bar{g} = g^{-1}$;
    \item Dimension: $d(g) = 1$ for all $g \in G$.
\end{itemize}
This example represents the simplest class of fusion algebras, where all 
basis elements have trivial dimension and the structural constants are given 
by the group multiplication.
\end{example}

\begin{example}[Representation ring]\label{ex:rep-ring}
Let $G$ be a compact matrix quantum group and let 
$\widehat{G}$ denote the set of equivalence classes of irreducible 
unitary corepresentations of $G$. The representation ring 
$\Xi = R(G) = \bigoplus_{\xi \in \widehat{G}} \mathbb{C}\xi$ 
carries a fusion algebra structure characterized by:
\begin{itemize}
    \item Basis: $I = \operatorname{Irr}(G)$, where $\operatorname{Irr}(G)$ is a 
    complete family of representatives for the equivalence classes of 
    irreducible unitary corepresentations of $G$;
    \item Multiplication: $N_{\xi,\eta}^{\zeta}$ equals the multiplicity of 
    $\zeta$ in the irreducible decomposition of the tensor product 
    $\xi \otimes \eta$;
    \item Conjugation: $\bar{\xi}$ is the conjugate corepresentation of $\xi$;
    \item Dimension: $d(\xi) = \dim \xi$, the quantum dimension.
\end{itemize}
The dimension condition (iv) in Definition~\ref{def:fusion_algebra} reflects 
the well-known identity $\dim(\xi \otimes \eta) = (\dim \xi)(\dim \eta)$ in 
representation theory.
\end{example}

The notion of amenability for finitely generated fusion algebras was 
introduced in \cite{14}. A finitely generated fusion algebra is called 
\emph{amenable} if there exist $1 < p < \infty$ and a finitely supported, 
symmetric, non‑degenerate probability measure $\mu$ on $I$ such that
\[
\|\lambda_{p,\mu}\| = 1.
\]

If $G$ is a compact matrix quantum group, its corepresentation ring 
$R(G)$ is finitely generated. Hence amenability of $G$ can 
be defined naturally through the amenability of $R(G)$.

In \cite{3}, David Kyed proposed a more general definition that applies to 
quantum groups that do not necessarily possess a fundamental corepresentation. 
We shall introduce this definition in detail below.

\begin{definition}[\cite{3}]\label{def:boundary}
Let $R = \mathbb{C}[I]$ be a fusion algebra. For two finite subsets 
$S, F \subseteq I$ we define the \emph{boundary of $F$ relative to $S$} as 
the set
\begin{align*}
\partial_S(F) &= \{\alpha \in F \mid \exists\, \xi \in S : 
\operatorname{supp}(\alpha\xi) \nsubseteq F\} \\
&\quad \cup \{\alpha \in F^c \mid \exists\, \xi \in S : 
\operatorname{supp}(\alpha\xi) \nsubseteq F^c\}.
\end{align*}
Here, and in what follows, $F^c$ denotes the set $I \setminus F$.
\end{definition}

\begin{definition}[\cite{3} F{\o}lner condition]\label{def:folner}
Let $R = \mathbb{C}[I]$ be a fusion algebra with dimension function $d$. 
For every finite, non-empty subset $S \subseteq I$ and every 
$\varepsilon > 0$ there exists a finite subset $F \subseteq I$ such that
\[
\sum_{\xi \in \partial_S(F)} d(\xi)^2 < \varepsilon 
\sum_{\xi \in F} d(\xi)^2.
\]
\end{definition}

\begin{remark}
By \cite{3}, a fusion algebra $R$ is amenable provided that it satisfies the 
F{\o}lner condition.
\end{remark}

\section{Compact Quantum Metric Space from Fusion Algebras}

This section studies compact quantum metric space structures on the fusion algebra
\(F(G)\) arising from higher-order commutators with the length-function
Dirac operator \(\tilde{D}_\ell\), and relates them to the corresponding
structures on \(C(G)\) and \(C_z(G)\). By Lemma~3.5 of~\cite{16},
\(\big(F(G),\ell^2(\operatorname{Irr}(G)),\tilde{D}_\ell\big)\) is a
spectral triple, and that polynomial growth of order \(s\) implies the rapid decay
property for \(F(G)\). Using Rieffel's criterion, we then
prove that for every integer \(k>s\), the pair \(\big(F(G),L_\ell^k\big)\)
is a compact quantum metric space, which applies in particular to
\(SU(N)\) and \(SO(N)\). Finally, for coamenable compact quantum groups of Kac
type, we show that the compact quantum metric property passes from
\(F(G)\) to \(C_z(G)\) and then to \(C(G)\)
, yielding the corresponding compact quantum metric space
structures on \(C(G)\) .

Let $G$ be a compact quantum group. The $\mathbb{C}$-vector space $F(G)$ with basis $\operatorname{Irr}(G)$ forms a $*$-algebra, called the fusion algebra.
We equip $F(G)$ with the inner product for which the elements of
$\operatorname{Irr}(G)$ form an orthonormal basis, and denote by
$\ell^{2}(\operatorname{Irr}(G))$ the completion.

Let $\{\alpha : \alpha \in \operatorname{Irr}(G)\}$ be the
orthogonality basis of $\ell^{2}(\operatorname{Irr}(G))$,
with
\[
\langle \alpha,\beta\rangle =\delta_{\alpha,\beta}.
\]
Here $d:\operatorname{Irr}(G)\to [1,\infty)$ is the dimension
function satisfying
\[
d(\alpha)=d(\bar{\alpha}), \qquad
d(\alpha)d(\beta)=\sum_{\gamma\in \operatorname{Irr}(G)}
N_{\alpha,\beta}^{\gamma}\, d(\gamma),
\qquad \alpha,\beta\in \operatorname{Irr}(G).
\]
We denote by $N_{\alpha,\beta}^{\gamma}$ the multiplicity of
$\gamma \in \operatorname{Irr}(G)$ in the decomposition of
$\alpha \otimes \beta$ into irreducibles; that is,
\begin{equation}
\alpha \otimes \beta \simeq
\bigoplus_{\gamma \in \operatorname{Irr}(G)}
\gamma^{\oplus N_{\alpha,\beta}^{\gamma}}.
\tag{3.1}
\end{equation}
Let $c_{c}(\operatorname{Irr}(G))$ denote the dense subspace of
$\ell^{2}(\operatorname{Irr}(G))$ spanned by this basis.

For each $f \in F(G)$, left multiplication defines a linear map
\[
\Lambda_{0}(f) \colon c_{c}(\operatorname{Irr}(G))
\longrightarrow c_{c}(\operatorname{Irr}(G)),
\]
and by \cite[Lemma~3.3]{16}, $\Lambda_{0}(f)$ extends to a bounded operator
$\Lambda(f)$ on $\ell^{2}(\operatorname{Irr}(G))$.
The connection between length functions and rapid‑decay quantum groups is established in~\cite{20} (see also~\cite{16}). We recall the corresponding definitions.

\begin{definition}(\cite{20})\label{def:length}
Let $G$ be a compact quantum group. A \emph{length function} on 
$G$ is a function 
$\ell \colon \operatorname{Irr}(G) \to [0,\infty)$ such that
\begin{enumerate}
    \item[(i)] $\ell(e) = 0$ ;
    \item[(ii)] $\ell(\bar{\alpha}) = \ell(\alpha)$ for all 
    $\alpha \in \operatorname{Irr}(G)$;
    \item[(iii)] $\ell(\gamma) \leq \ell(\alpha) + \ell(\beta)$ for all 
    $\alpha, \beta, \gamma \in \operatorname{Irr}(G)$ such that 
    $\gamma$ is equivalent to a sub-corepresentation of 
    $\alpha \otimes \beta$.
\end{enumerate}
\end{definition}

A length function $\ell$ is called \emph{proper} if 
$\ell^{-1}([0,N])$ is finite for all $N \geq 0$ and 
$\ell(\alpha) = 0$ only when $\alpha = e$.
The length function also defines an unbounded operator 
\[
\tilde{D}_\ell : c_c(\operatorname{Irr}(G)) \longrightarrow 
\ell^2(\operatorname{Irr}(G)),\qquad 
\tilde{D}_\ell(\alpha) = \ell(\alpha)\,\alpha.
\]
Assume moreover that $\ell$ is integer-valued.
Thus, assume that $G$ is equipped with a proper 
length function $\ell \colon \operatorname{Irr}(G) \to 
\mathbb{N}$. Setting 
\[
A_{n} := \operatorname{span}\{\alpha \in \operatorname{Irr}(G) 
\mid \ell(\alpha) \leqslant n\},
\]
then $A_m \subset A_n$ if $m < n$, 
$F(G) = \bigcup_{n=0}^{\infty} A_n$, $A_n^* = A_n$, 
$A_m A_n \subseteq A_{m+n}$, and $A_0 = \mathbb{C}e$ where $e$ is the unit element. Consequently, 
the fusion algebra $F(G)$ belongs to the class of filtered 
$*$-algebras studied by Ozawa and Rieffel in \cite{18}.

We can view each $A_n$ as a finite-dimensional, hence closed, subspace 
of $\mathcal{H}$. Let $P_n$ denote the orthogonal projection of 
$\mathcal{H}$ onto $A_n$. We set $Q_0 = P_0$ and $Q_n = P_n - P_{n-1}$ for $n \ge 1$. Then, for each $a \in F(G)$, we define
\[
a_n = Q_n(a),
\]
so that $\tilde{D}_\ell(a_n) = n a_n$. In particular, if there is no $\alpha \in \operatorname{Irr}(G)$ such that $\ell(\alpha)=n$, then $a_n = 0$.
 The $Q_n$ are mutually orthogonal, and 
$\sum_{n=0}^\infty Q_n = I_{\mathcal{H}}$ in the strong operator 
topology. Then $a_n \in A_n$, but $a_n \notin A_{n-1}$ unless 
$a_n = 0$. Moreover, $a = \sum_n a_n$, with at most $p$ non‑zero terms 
if $a \in A_p$.

In fact, $\tilde{D}_\ell = \sum_{n=1}^\infty n Q_n$. We now consider whether it defines a spectral triple for the fusion algebra.

\begin{lemma}[\cite{16} Lemma 3.5]
Let $G$ be a compact quantum group equipped with a proper 
length function $\ell$. Then the triple 
$\bigl(F(G), \ell^2(\operatorname{Irr}(G)), 
\tilde{D}_\ell\bigr)$ is a spectral triple.
\end{lemma}

When considering a discrete group equipped with a length function, one often studies the properties of rapid decay and polynomial growth. In \cite{20}, Vergnioux introduced a notion of rapid decay for discrete quantum groups. We will apply this property of rapid decay to fusion algebras constructed from the irreducible corepresentations of a compact quantum group. Since every $f \in F(G)$ has finite support, we can directly give a definition on the fusion algebra.

\begin{definition}[Rapid Decay for Fusion Algebras]\label{rapid decay}
Let $G$ be a compact  quantum group.Fix a length function \(\ell : \operatorname{Irr}(G) \to [0,\infty)\). For \(f \in F(G)\) define the weighted \(\ell^2\)-norm
\[
\|f\|_{2,s} := \left( \sum_{\alpha \in \operatorname{Irr}(G)} |f(\alpha)|^2 \, (1+\ell(\alpha))^{2s} \right)^{1/2} \qquad s>0.
\]
 The fusion algebra \(F(G)\) is said to have the \emph{rapid decay property} if there exist constants \(C>0\) and \(s>0\) such that
\[
\|\Lambda(f)\| \le C \|f\|_{2,s}
\]
for every \(f \in F(G)\).
\end{definition}

In \cite{20}, Vergnioux introduced the notion of polynomial growth for discrete quantum groups. Here we give the definition of polynomial growth for compact quantum groups.

\begin{definition}\label{polynomial growth}
    Let \(G\) be a compact quantum group. We say that \(G\) has
    \emph{polynomial growth of order \(s\) with respect to a length function \(\ell\) on \(G\)}
    if there exists a constant \(c\in\mathbb{R}_+\) such that
    \[
    \sum_{\substack{\ell(\alpha)\in(n-1,n]}} d(\alpha)^2 \leq c n^s
    \]
    for all \(n\in\mathbb{N}\).

    We say that \(G\) has \emph{strong polynomial growth of order \(s\) with respect to a length function \(\ell\) on \(G\)}
    if there exist constants \(c_1,c_2\in\mathbb{R}_+\) such that
    \[
    c_2 n^s \leq \sum_{\substack{\ell(\alpha)\in(n-1,n]}} d(\alpha)^2 \leq c_1 n^s
    \]
    for all \(n\in\mathbb{N}\).

    We say that \(G\) has \emph{(strong) polynomial growth} if it has
    (strong) polynomial growth with respect to some length function on \(G\).
\end{definition}

 Under this definition, we obtain the relationship between polynomial growth and rapid decay in the fusion‑algebra framework.

\begin{proposition}\label{prposition rapid decay}
    Let \(G\) be a compact quantum group with polynomial growth . 
    Then \(F(G)\) has the rapid decay property.
\end{proposition}

\begin{proof}
    Let \(f=\sum_{\alpha\in \operatorname{Irr}(G)} f(\alpha)\,\alpha \in F(G)\), 
    and let \(\Lambda\) denote the left regular representation of \(F(G)\) on 
    \(\ell^2(\operatorname{Irr}(G))\). 
    For a unit vector \(\xi=\sum_{\beta}\xi_\beta \beta\), we have
    \[
        \Lambda(f)\xi
        = \sum_{\alpha} f(\alpha)\Lambda(\alpha)\xi
        = \sum_{\alpha,\beta,\gamma} f(\alpha)\xi_\beta N_{\alpha,\beta}^{\gamma}\,\gamma,
    \]
    where \(N_{\alpha,\beta}^{\gamma}\) are the fusion coefficients. Hence,
    \begin{align*}
        \|\Lambda(f)\|
        &= \sup_{\|\xi\|=1}\|\Lambda(f)\xi\| \\
        &\leq \sum_{\alpha} |f(\alpha)|\,
        \sup_{\|\xi\|=1}\left\|
            \sum_{\beta,\gamma} \xi_\beta N_{\alpha,\beta}^{\gamma}\,\gamma
        \right\|.
    \end{align*}

   If \(G\) has polynomial growth of order \(s\), there exists a constant \(C_0>0\) such that for every \(k>s\),
    \[
        \sum_{\alpha\in \operatorname{Irr}(G)}
        \frac{d(\alpha)^2}{(1+\ell(\alpha))^{2k}} < \infty.
    \]
    Moreover, the fusion rules imply the standard dimension estimate
    \[
        \sum_{\gamma} N_{\alpha,\beta}^{\gamma}\, d(\gamma) = d(\alpha)d(\beta),
    \]
    and Frobenius reciprocity constants obey the symmetry relations
    \[
    N_{\alpha,\beta}^\gamma=N_{\bar{\alpha},\gamma}^{\beta}=N_{\gamma,\bar{\beta}}^{\alpha} .      
    \]
      By Cauchy--Schwarz,
    \begin{align*}
        \sup_{\|\xi\|=1}
        \left\|
            \sum_{\beta,\gamma} \xi_\beta N_{\alpha,\beta}^{\gamma}\,\gamma
        \right\|
        &=
        \sup_{\|\xi\|=1}
        \left(
            \sum_{\gamma}
            \left(\sum_{\beta}  (\frac{N_{\alpha,\beta}^{\gamma}d(\beta)}{d(\gamma)})^{\frac{1}{2}} (\frac{N_{\alpha,\beta}^{\gamma}d(\gamma)}{d(\beta)})^{\frac{1}{2}} |\xi_\beta|\right)^2
           \right)^{1/2} \\
        &\leq
        \sup_{\|\xi\|=1}
        \left(
            \sum_{\gamma}
            \left(\sum_{\beta}  \frac{N_{\bar{\alpha},\gamma}^{\beta}d(\beta)}{d(\gamma)} \sum_{\beta}\frac{N_{\alpha,\beta}^{\gamma}d(\gamma)}{d(\beta)}|\xi_\beta|^2 \right)
           \right)^{1/2}
         \\
        &= 
        \sup_{\|\xi\|=1}
        \left(d(\bar{\alpha})\sum_{\beta} |\xi_\beta|^2 \sum_{\gamma}\frac{N_{\alpha,\beta}^{\gamma}d(\gamma)}{d(\beta)}
        \right)^{1/2} \\
        &= d(\alpha)\\
        \end{align*}
    
    Applying Cauchy--Schwarz once more, we obtain
    \begin{align*}
        \|\Lambda(f)\|
        &\leq
        \left(
            \sum_{\alpha}
            |f(\alpha)|^2  (1+\ell(\alpha))^{2k}
        \right)^{1/2}
        \left(
            \sum_{\alpha}
            \frac{d(\alpha)^2}{(1+\ell(\alpha))^{2k}}
        \right)^{1/2}.
    \end{align*}
    Since \(k>s\), the second factor is finite by polynomial growth, and hence there exists a constant
    \(C_k>0\) such that
    \[
        \|\Lambda(f)\|
        \leq
        C_k
        \left(
            \sum_{\alpha}
            |f(\alpha)|^2 (1+\ell(\alpha))^{2k}
        \right)^{1/2}.
    \]
    Therefore \(F(G)\) has rapid decay.
\end{proof}

We now consider Lipschitz seminorms on the fusion algebra.
If we define \(\delta(f):=[\tilde{D}_\ell,\Lambda(f)]\), then 
\(\delta^k(f)= [\tilde{D}_\ell, [\tilde{D}_\ell, \dots [\tilde{D}_\ell,\Lambda(f)]\dots ]]\),
where the commutator is taken \(k\) times.

\begin{lemma}
    For the fusion algebra \(F(G)\) and for every positive integer 
    \(k\), \(L_\ell^k(f)=\| \delta^k(f)\|
\) is a Lipschitz seminorm.
\end{lemma}
\begin{proof}
    For any \(f = \sum_{\alpha} f(\alpha)\,\alpha \in F(G)\), let 
    \(\Lambda(f)\) denote its left regular representation on 
    \(\ell^2(\operatorname{Irr}(G))\):
    \[
    \Lambda(f)\beta = \sum_{\alpha,\gamma\in\operatorname{Irr}(G)} 
    f(\alpha) N_{\alpha,\beta}^{\gamma} \,\gamma.
    \]

   The adjoint \([\tilde{D}_\ell,\Lambda(f)]^*\) is given by
\[
\langle \beta, [\tilde{D}_\ell,\Lambda(f)]^* \gamma \rangle
= \sum_{\alpha}(\ell(\gamma)-\ell(\beta)) \overline{f(\alpha)} 
  N_{\alpha,\beta}^{\gamma}.
\]

On the other hand,
\begin{align*}
\langle \beta, [\tilde{D}_\ell,\Lambda(f^*)] \gamma \rangle
&= (\ell(\beta)-\ell(\gamma)) \langle \beta,\Lambda(f^*)\gamma \rangle \\
&= (\ell(\beta)-\ell(\gamma)) \sum_{\alpha} f^*(\alpha) 
   N_{\bar\alpha,\gamma}^{\beta} \\
&= \sum_{\alpha} (\ell(\beta)-\ell(\gamma)) \overline{f(\alpha)} 
   N_{\bar\alpha,\gamma}^{\beta}.
\end{align*}

For a compact quantum group, the fusion coefficients satisfy
\[
N_{\bar\alpha,\gamma}^{\beta}=N_{\alpha,\beta}^{\gamma}.
\]
Combining this with the preceding identity, we get
\[
\begin{aligned}
\big\langle \beta,\,[\tilde{D}_{\ell},\Lambda(f^{*})]\gamma\big\rangle
&=\sum_{\alpha}\big(\ell(\beta)-\ell(\gamma)\big)\,\overline{f(\alpha)}\,
N_{\alpha,\beta}^{\gamma}\\
&=-\sum_{\alpha}\big(\ell(\gamma)-\ell(\beta)\big)\,\overline{f(\alpha)}\,
N_{\alpha,\beta}^{\gamma}.
\end{aligned}
\]
Consequently,
\[
[\tilde{D}_\ell,\Lambda(f^*)]=-[\tilde{D}_\ell,\Lambda(f)]^*,
\]
or equivalently,
\[
[\tilde{D}_\ell,\Lambda(f)]^*=-[\tilde{D}_\ell,\Lambda(f^*)].
\]

    Comparing with the matrix elements of \([\tilde{D}_\ell,\Lambda(f)]^*\) we obtain
    \[
    [\tilde{D}_\ell,\Lambda(f)]^* = -\,[\tilde{D}_\ell,\Lambda(f^*)].
    \]
    If we define \(\delta(f):=[\tilde{D}_\ell,\Lambda(f)]\), then 
    \(\delta^k(f)= [\tilde{D}_\ell, [\tilde{D}_\ell, \dots 
    [\tilde{D}_\ell,\Lambda(f)]\dots ]]\).
    Now, for \(k \in \mathbb{N}\) and \(f \in  F(G)\), consider
    \[
    \delta^k(f)^* = (-1)^k \delta^k(f^*) .
    \]
    Then \(\|\delta^k(f)^*\| = \|\delta^k(f^*)\|\).
    Note that if \(e\in\operatorname{Irr}(G)\) is the identity of 
    \(F(G)\), 
    \(\delta^k(f)e = \sum_{\alpha\in\operatorname{Irr}(G)} 
    \ell(\alpha)^k f(\alpha)\alpha\).
    Then for \(f\in F(G)\), we have \(\|\delta^k(f)\|=0\) if and only if 
    \(f\in \mathbb{C}e\).
    Thus, the Lipschitz seminorm can be defined as the 
    operator norm of this commutator:
    \[
    L_k(\cdot) := \|\delta^k(\cdot)\|.
    \]
\end{proof}

For any Lipschitz seminorm \(\|\cdot\|\) on \(F(G)\), \(f \in F(G)\) 
and \(r \geqslant 0\) we denote the corresponding closed balls as follows:
\[
\mathbb{B}_r^{\|\cdot\|}(f) := \{g \in F(G) \mid \|f - g\| \leqslant r\} .
\]
Here we let \([\cdot]: F(G) \to F(G)/\mathbb{C}\) denote the 
quotient map and \(\|\cdot\|_{F(G)/\mathbb{C}}\) denote the quotient 
norm on \(F(G)/\mathbb{C}\).

Following the strategy used in~\cite{23} to equip discrete quantum groups with the rapid decay property with a Lip-norm, we prove the following theorem.

\begin{theorem}
\label{compact quantum metric space on fusion algebra}
Let $G$ be a compact quantum group equipped
with a proper length function $\ell$. Suppose that
$G$ has polynomial growth of order $s\geq 0$.
Then, for every positive integer $k>s$, the seminorm
$L_\ell^k(f)$ defines a compact quantum metric space structure
on the operator-norm closure of $F(G)$.
\end{theorem}

\begin{proof}
Let $A$ denote the operator-norm closure of
$\Lambda(F(G))$ in
$B(\ell^2(\operatorname{Irr}(G)))$.

By Rieffel's criterion~\cite{9}, it suffices to
show that the image of
\[
B=\{f\in F(G):L_\ell^k(f)\leq 1\}
\]
in $A/\mathbb{C}1$ is totally bounded.

Since $L_\ell^k$ vanishes on scalars, we may
assume that $f(e)=0$ for every $f\in B$.

For any such $f$, we have
\[
\begin{aligned}
\left\|\delta^k
(f)e\right\|_2^2
&=
\sum_{\alpha\in\operatorname{Irr}(G)}
|f(\alpha)|^2\ell(\alpha)^{2k}\\
&\leq L_\ell^k(f)^2
\leq 1.
\end{aligned}
\]

For $n\geq 1$, write
\[
f=f_{\leq n}+f_{>n},
\]
where
\[
f_{\leq n}
=
\sum_{0<\ell(\alpha)\leq n}f(\alpha)\alpha,
\qquad
f_{>n}
=
\sum_{\ell(\alpha)>n}f(\alpha)\alpha.
\]

For the low-frequency part, using
$\|\Lambda(\alpha)\|\leq d(\alpha)$ and
the Cauchy--Schwarz inequality, we obtain
\[
\begin{aligned}
\|\Lambda(f_{\leq n})\|
&\leq
\sum_{0<\ell(\alpha)\leq n}
|f(\alpha)|d(\alpha)\\
&\le \biggl( \sum_{\ell(\alpha)\le n} |f(\alpha)|^2  
           \ell(\alpha)^{2k}) \biggr)^{1/2}
           \biggl( \sum_{\ell(\alpha)\le n} 
           \frac{d(\alpha)^2}{\ell(\alpha)^{2k}} \biggr)^{1/2} \\
        &\le C_n L_\ell^k(f)\\
         &\le C_n,
\end{aligned}
\]
where
\[
C_n=
\left(
\sum_{0<\ell(\alpha)\leq n}
\frac{d(\alpha)^2}{\ell(\alpha)^{2k}}
\right)^{1/2}<\infty.
\]
Thus, for each fixed $n$, the low-frequency
parts form a bounded subset of a
finite-dimensional space and hence are
totally bounded.

Next, choose $r$ such that
\[
\frac{s+1}{2}<r<k.
\]
Such an $r$ exists since $k$ is a positive
integer and $k>s\geq 0$.

By polynomial growth,
\[
\begin{aligned}
\sum_{\alpha\in\operatorname{Irr}(G)}
\frac{d(\alpha)^2}{(1+\ell(\alpha))^{2r}}
&\leq
1+C\sum_{m=1}^{\infty}m^{s-2r}\\
&<\infty.
\end{aligned}
\]
Consequently, the rapid decay estimate holds
with exponent $r$: there exists $C_r>0$
such that
\[
\|\Lambda(g)\|
\leq C_r\|g\|_{2,r},
\qquad g\in F(G).
\]

For the high-frequency part, we therefore have
\[
\begin{aligned}
\|\Lambda(f_{>n})\|
&\leq
C_r\left(
\sum_{\ell(\alpha)>n}
|f(\alpha)|^2
(1+\ell(\alpha))^{2r}
\right)^{1/2}\\
&\leq
C_r2^r n^{r-k}
\left(
\sum_{\ell(\alpha)>n}
|f(\alpha)|^2
\ell(\alpha)^{2k}
\right)^{1/2}\\
&\leq C_r2^r n^{r-k}.
\end{aligned}
\]
Since $r<k$, it follows that
\[
\sup_{f\in B}
\|\Lambda(f_{>n})\|
\longrightarrow 0
\qquad (n\to\infty).
\]

Thus, the image of $B$ in
$A/\mathbb{C}1$ can be uniformly approximated
by bounded subsets of finite-dimensional
spaces. Hence it is totally bounded.

By Rieffel's criterion, $L_\ell^k$ induces
the weak$^*$ topology on the state space
of $A$. Therefore, $(A,L_\ell^k)$ is a
compact quantum metric space, where the
seminorm is understood on the dense
$*$-subalgebra $F(G)$.
\end{proof}

This theorem is related to the work of~\cite{16} and shows that fusion algebras satisfying the stated assumptions give rise to compact quantum metric spaces. As a consequence, coamenable compact quantum groups of Kac type can be endowed with the same Lip‑norm to yield compact quantum metric spaces.

\begin{example}[Connected compact Lie groups]
Let $G$ be a nontrivial connected compact Lie group
of dimension $D$, and let $T$ be a maximal torus of
dimension $r$. Denote by $\Phi^+$ the set of positive
roots of $G$, and put $m=|\Phi^+|$. Recall that
\[
    D=r+2m.
\]

For each $\alpha\in\operatorname{Irr}(G)$, let
$\lambda_\alpha$ denote its dominant integral highest
weight. Fix a Weyl-invariant norm $\|\cdot\|$ on the
real weight space and define
\[
    \ell(\alpha)=\|\lambda_\alpha\|.
\]
The highest-weight description of irreducible
representations and the properties of Weyl-invariant
norms imply that $\ell$ is a proper length function.

By the Weyl dimension formula
(see \cite{FultonHarris1991}), we have
\[
    d(\alpha)
    =
    \prod_{\beta\in\Phi^+}
    \frac{\langle\lambda_\alpha+\rho,\beta^\vee\rangle}
         {\langle\rho,\beta^\vee\rangle},
\]
where $\rho$ is the half-sum of the positive roots
and $\beta^\vee$ denotes the coroot associated
with $\beta$.

Since $\Phi^+$ is finite, there exists a constant
$C>0$ such that
\[
    d(\alpha)
    \leq C(1+\ell(\alpha))^m.
\]
Moreover, the dominant integral weights form a
subset of a lattice of rank $r$. Standard lattice-point
estimates imply that
\[
    \#\bigl\{
        \alpha\in\operatorname{Irr}(G):
        n-1<\ell(\alpha)\leq n
    \bigr\}
    =O(n^{r-1}).
\]
Consequently,
\begin{align*}
    \sum_{n-1<\ell(\alpha)\leq n}d(\alpha)^2
    &\leq
    C^2(1+n)^{2m}
    \#\bigl\{
        \alpha\in\operatorname{Irr}(G):
        n-1<\ell(\alpha)\leq n
    \bigr\} \\
    &=O(n^{r+2m-1})\\
    &=O(n^{D-1}).
\end{align*}
Thus, $G$ has polynomial growth of order $D-1$
in the sense of Definition~3.4.

In particular, this applies to the classical compact
Lie groups $\mathrm{SU}(N)$ for $N\geq 2$ and
$\mathrm{SO}(N)$ for $N\geq 3$ .(see \cite{{FultonHarris1991}})

Therefore, by
Theorem~\ref{compact quantum metric space on fusion algebra},
for every integer $k>D-1$, the seminorm $L_\ell^k$
induces a compact quantum metric space structure
on the operator-norm closure of $F(G)$.
\end{example}

Let $G$ be a compact quantum group with Hopf algebra of polynomial 
functions $\operatorname{Pol}(G)$. Let $\widehat{G}$ be the 
set of equivalence classes of irreducible unitary representations.
Let $\{u_{ij}^\alpha\}$ be the matrix coefficients of 
$\alpha \in \widehat{G}$, and let $(\mathcal{H}, \Lambda)$ denote the 
GNS representation of $C(G)$ associated with the Haar state. Then 
$\{\Lambda(u_{ij}^\alpha)\}$ forms a basis of $\mathcal{H}$.
According to Definition~\ref{def:length}, for a proper length function, an 
unbounded Dirac operator 
$D_\ell:\Lambda(\operatorname{Pol}(G))\to \mathcal{H}$ can be defined by
\[
D_\ell\bigl(\Lambda(u_{i,j}^\alpha)\bigr) = \ell(\alpha)\,\Lambda(u_{i,j}^\alpha) .
\]
By \cite[Lemma~5.4]{23}, we know that 
$(\operatorname{Pol}(G),\mathcal{H},D_\ell)$ is a spectral triple. 

 Now we begin to consider Lipschitz seminorm  on compact quantum groups.
The commutator seminorm 
$L_\ell^k: \operatorname{Pol}(G) \to [0,\infty)$ is defined by
\[
L_\ell^k(a) := \|[\cdots[D_\ell,[D_\ell, a]]\cdots]\|,
\]
where $[D_\ell, a] = D_\ell a - a D_\ell$ , the commutator is taken $k$ times and $\|\cdot\|$ is the operator norm 
on $L^2(G)$.

Then  $L_\ell^k$ is obviously a Lipschitz seminorm  in the sense of Rieffel.

\begin{lemma}\label{lemma:Lipschitz norm }
   For any $k\in \mathbb{N}$ ,the seminorm $L_\ell^k$ is a Lipschitz seminorm  on $C(G)$.
\end{lemma}

\begin{proof}
   The  proof follows essentially the same argument as in Lemma~6.4 of \cite{23}, with only the object of study being replaced.
\end{proof}

For coamenable compact quantum groups of Kac type, it follows from \cite{16} 
that $(C(G), L_\ell)$ is a compact quantum metric space if and only 
if $(F(G), L_\ell)$ is a compact quantum metric space. Later , we will show that $(C(G), L_\ell^k)$ is a compact quantum 
metric space if and only if $(F(G), L_\ell^k)$ is a compact quantum 
metric space. 

We first recall the definition of the algebra of central functions,
\begin{equation}\label{eq:central-functions}
    C_z(G) := \{a \in C(G) \mid \sigma\Delta(a) = \Delta(a)\},
\end{equation}
where $\sigma$ denotes the flip automorphism on $C(G)\otimes C(G)$.
This algebra characterizes precisely the conditions under which 
$\bigl(C(G),\, L_\ell \bigr)$, where $L_\ell$ denotes the Lipschitz seminorm  
induced by the length function $\ell$, forms a \emph{compact quantum metric space} 
in the sense of Rieffel.

Let $\Lambda$ be the GNS representation of $C(G)$ on $L^2(G)$. 
By construction, the unbounded operator $D_\ell$ on $L^2(G)$ leaves the subspace $\Lambda(\operatorname{Pol}_z(G))$ 
invariant. Its restriction to this subspace is essentially self‑adjoint; we 
denote its self‑adjoint closure (still written as $D_\ell^z$) by the same symbol. 
Consequently $(\operatorname{Pol}_z(G), L^2_z(G), D_\ell^z)$ 
becomes a spectral triple, where 
$L^2_z(G) := \overline{\Lambda(\operatorname{Pol}_z(G))}$.
We  recall that a state $\varphi\in S(C(G))$ is called \emph{central} if 
    $(\varphi \otimes \mathrm{id})\Delta = (\mathrm{id}\otimes \varphi)\Delta$.

\begin{remark}
    In \cite{16}, Austad and Kyed defined a normal conditional expectation 
    $E:\mathbb{B}(L^2(G))\to \mathbb{B}(L^2(G))$ given by 
    $E(a)=(h\otimes \mathrm{id})\partial(a)$, where 
    $\partial(\cdot):=Z^*(1\otimes \cdot)Z$. And they prove that a state 
    $\varphi\in S(C(G))$ is central if and only if 
    $\varphi=\varphi \circ E$ and there exists a sequence of central states 
    $\{\chi_n\}_{n\in \mathbb{N}}$ where are supported in only finitely many matrix coefficients converging to $\epsilon$ in the weak$^*$ 
    topology on $S(C(G))$.
\end{remark}

For coamenable compact quantum groups of Kac type, Austad and Kyed established a characterization of compact quantum metric spaces in terms of the central subalgebra; see~\cite[Theorem A]{16}. We establish a corresponding one-sided result for the higher-order commutator seminorms \(L_\ell^k\). The proof follows the conditional-expectation approach of~\cite{16}, together with an additional verification that the relevant conditional expectation is contractive with respect to the higher-order commutator seminorms.

\begin{theorem}\label{compact quantum metric space} 
 If $G$ is a coamenable compact quantum group of Kac type and 
 $\ell:\operatorname{Irr}(G)\to[0,\infty)$ is a proper length 
 function, then $(C(G),L_\ell^k)$ is a compact quantum metric space 
 if and only if $(C_z(G),L_\ell^k)$ is.
\end{theorem}

\begin{proof}
By \cite{16} lemma 4.3,  the operator $\mathrm{id} \otimes D_\ell$ commutes with $W,W^*$
 $\Sigma V \Sigma$, $\Sigma V^* \Sigma$, $Z$ and $Z^*$, as operators on the 
 dense subspace 
 $\Lambda(\operatorname{Pol}(G))\odot\Lambda(\operatorname{Pol}(G))
 \subset L^2(G)\hat{\otimes}L^2(G)$.

 For $a,x,y \in \operatorname{Pol}(G)$, consider the relation between 
 $L_\ell^k(E(a))$ and $L_\ell^k(a)$:

Since \(Z\) commutes with \(1\otimes D_\ell\), for any
\(x,y\in\operatorname{Pol}(\mathbb G)\), we obtain
\begin{align*}
&\bigl\langle \Lambda(x),
    \delta^k(E(a))\Lambda(y)\bigr\rangle \\
&=\sum_{l=0}^{k}(-1)^{k-l}\binom{k}{l}
  \bigl\langle \Lambda(x),
    D_\ell^l E(a)D_\ell^{k-l}\Lambda(y)
  \bigr\rangle\\
&=\sum_{l=0}^{k}(-1)^{k-l}\binom{k}{l}
  \bigl\langle
    \Lambda(1)\otimes D_\ell^l\Lambda(x),
    Z^*(1\otimes a)Z
    \bigl(\Lambda(1)\otimes
    D_\ell^{k-l}\Lambda(y)\bigr)
  \bigr\rangle\\
&=\sum_{l=0}^{k}(-1)^{k-l}\binom{k}{l}
  \bigl\langle
    \Lambda(1)\otimes\Lambda(x),
    Z^*\bigl(1\otimes
    D_\ell^l a D_\ell^{k-l}\bigr)Z
    \bigl(\Lambda(1)\otimes\Lambda(y)\bigr)
  \bigr\rangle\\
&=\bigl\langle
    \Lambda(1)\otimes\Lambda(x),
    Z^*\bigl(1\otimes\delta^k(a)\bigr)Z
    \bigl(\Lambda(1)\otimes\Lambda(y)\bigr)
  \bigr\rangle\\
&=\bigl\langle
    \Lambda(x),E(\delta^k(a))\Lambda(y)
  \bigr\rangle.
\end{align*}
Then $\delta^k(E(a)) = E(\delta^k(a))$ and by \cite[Proposition~4.2]{16} we 
 know the map defines a normal conditional expectation, hence 
 $L_\ell^k(E(a)) = \|\delta^k(E(a))\| = \|E(\delta^k(a))\| \le L_\ell^k(a)$.

 Since for any two central states $\varphi$ and $\psi$ on $C(G)$, 
 $\varphi = \varphi \circ E$ and $\psi = \psi \circ E$. When $G$ is 
 of Kac type, by \cite[Proposition~4.2]{16} we know $E$ maps 
 $\operatorname{Pol}(G)$ to $\operatorname{Pol}_z(G)$. Then
\begin{align*}
    d_\ell(\varphi, \psi) &:= \sup \bigl\{ |\varphi(a) - \psi(a)| \; ; \; 
        a \in \operatorname{Pol}(G), \; L_\ell^k(a) \leq 1 \bigr\} \\
    &= \sup \bigl\{ |\varphi(a) - \psi(a)| \; ; \; 
        a \in \operatorname{Pol}_z(G), \; L_\ell^k(a) \leq 1 \bigr\} \\
    &=: d_\ell^z \bigl( \varphi|_{C_z(G)}, \; \psi|_{C_z(G)} \bigr).
\end{align*}

 For any $\chi\in S(C(G))$ and for any $a\in C(G)$, we denote 
 by $\alpha_\chi:C(G)\to C(G)$ the ucp map defined by 
 $\alpha_\chi(a)=(\mathrm{id}\otimes \chi)\Delta(a)$.
\begin{align*}
    \|a-\alpha_\chi(a)\|
    &=\sup\{\,|\varphi(a-\alpha_\chi(a))| \mid \|\varphi\|\leq 1\,\}  \\
    &=\sup\{\,|\varphi(\mathrm{id}\otimes(\epsilon-\chi))\Delta a| 
        \mid \|\varphi\|\leq 1\,\}\\
    &\leq \sup\{\,d_\ell(\epsilon,\chi) \cdot L_\ell^k((\varphi\otimes\mathrm{id})\Delta a) 
        \mid \|\varphi\|\leq 1\,\}\\
    &\leq d_\ell(\epsilon,\chi)\cdot L_\ell^k(a).
\end{align*}
We will use the bi-invariance of $L_\ell^k$ established in Lemma~4.2 below; its proof is independent of the present section.
 By Remark~3.10, there exists a sequence of central states 
 $\chi_n$ converging to $\epsilon$, hence we obtain:
\begin{align*}
    \|a-\alpha_{\chi_n}(a)\|
    &=\|\alpha_\epsilon(a)-\alpha_{\chi_n}(a)\| \\
    &\leq d_\ell(\epsilon,\chi_n)\cdot L_\ell^k(a)
     = d_\ell\bigl(\epsilon|_{C_z(G)},\chi_n|_{C_z(G)}\bigr)
       \cdot L_\ell^k(a).
\end{align*}
Since each $\alpha_{\chi_n}$ is a ucp map with finite‑dimensional image and
\[
\lim_{n\to\infty} d_\ell^z\bigl(\epsilon|_{C_z(G)},\chi_n|_{C_z(G)}\bigr)=0,
\]
the conclusion follows from \cite[Theorem~3.1]{27}.
\end{proof}

\begin{lemma}\label{lemma fusion algebra and Cz}
   If \((F(G),L_\ell^k)\) is a compact quantum metric space, then the
restriction of \(L_\ell^k\) to \(C_z(G)\) endows \(C_z(G)\)
with the structure of a compact quantum metric space.
\end{lemma}

\begin{proof}
For any \(a\in \operatorname{Pol}_z(G)\), we have the following estimate:
\begin{align*}
    \operatorname{L}_\ell^k(a)
    &= \sup\bigl\{\|\delta^k(a)\xi\,\big\|\,\big|\, \xi\in \bigl(\Lambda \operatorname{Pol}(G)\bigr)_1\bigr\} \\
    &\geq \sup\bigl\{\|\delta^k(a)\xi\,\big\|\,\big|\, \xi\in \bigl(\Lambda \operatorname{Pol}_z(G)\bigr)_1\bigr\} \\
    &= \|\delta_z^k(a)\|.
\end{align*}
Here the higher-order commutator \(\delta_z^k(a)\) is defined iteratively by
\[
\delta_z^k(a) = \bigl[D_\ell^z,\bigl[D_\ell^z,\cdots \bigl[D_\ell^z,a\bigr]\cdots \bigr]\bigr],
\]
where the commutator operation is applied \(k\) times. By Proposition 3.6 in \cite{16}, the spectral triples
\[
\bigl(F(G),\ell^2(\operatorname{Irr}(G)),\tilde{D}_{\ell}\bigr)
\quad \text{and} \quad
\bigl(\operatorname{Pol}_z(G),\ell^2(\operatorname{Irr}(G)),D_{\ell}^z\bigr)
\]
are unitarily equivalent.

Assume that \((F(G),\operatorname{L}_{\ell}^k)\) is a compact quantum metric space. It then follows that the pair \((C_z(G)),\|\delta_z^k(\cdot)\|\bigr)\) also forms a compact quantum metric space. Applying Theorem 1.8 from \cite{32}, we deduce that the set
\[
\bigl\{a\in \operatorname{Pol}_z(G) \,\big|\, \|\delta_z^k(a)\|\leq 1\bigr\}
\]
is precompact in the quotient space \(C_z(G)/\mathbb{C}\).

From the preceding inequality \(\operatorname{L}_\ell^k(a)\geq \|\delta_z^k(a)\|\), we obtain the set inclusion
\[
\bigl\{a\in \operatorname{Pol}_z(G) \,\big|\, \operatorname{L}_\ell^k(a)\leq 1\bigr\}
\subseteq \bigl\{a\in \operatorname{Pol}_z(G) \,\big|\, \|\delta_z^k(a)\|\leq 1\bigr\}.
\]
As a subset of a precompact set in \(C_z(G)/\mathbb{C}\), the set \(\{a\in \operatorname{Pol}_z(G) \mid \operatorname{L}_\ell^k(a)\leq 1\}\) is also precompact in \(C_z(G)/\mathbb{C}\). Consequently, \((C_z(G),\operatorname{L}_\ell^k)\) is a compact quantum metric space, which completes the proof.
\end{proof}

\begin{theorem}\label{theorem fusion and C(G)}
  If $G$ is a coamenable compact quantum group of Kac type and $(F(G),L_\ell^k)$ is a compact quantum metric space, then $(C(G),L_\ell^k)$ is a compact quantum metric space.
\end{theorem}
\begin{proof}
    The conclusion follows immediately from Theorem \ref{compact quantum metric space} and Lemma \ref{lemma fusion algebra and Cz}.
\end{proof}

\begin{corollary}\label{coro polynoimal forms cqms}
If $G$ is a coamenable Kac-type compact quantum group with polynomial growth of order $s$, then for any positive integer $k>s$, the algebra $C(G)$ equipped with the seminorm $\operatorname{L}_\ell^k$ forms a compact quantum metric space.
\end{corollary}
\begin{proof}
The conclusion follows immediately from Theorem \ref{compact quantum metric space on fusion algebra} and Theorem \ref{theorem fusion and C(G)}.
\end{proof}

 For every $N\in\mathbb{N}$, $\mathrm{SU}(N)$ and $\mathrm{SO}(N)$ (see \cite{{FultonHarris1991}})have  polynomial growth. Therefore, there exists a positive integer $k$ such that these compact quantum groups, equipped with the Lipschitz seminorm $\mathrm{L}_\ell^k$, form compact quantum metric spaces.

\section{Gromov--Hausdorff Convergence of Spectral
Truncation for Compact Quantum Groups}

This section establishes the quantum Gromov--Hausdorff convergence of spectral
truncations of \(C(G)\) for compact quantum groups equipped with a proper
length function. For the Dirac operator \(D_\ell\) associated with \(\ell\), let
\(P_\Lambda\) be the spectral projection onto the eigenspaces with eigenvalues of
modulus at most \(\Lambda\), and consider the truncated space
\[
  C(G)_\Lambda := P_\Lambda C(G)P_\Lambda
\]
acting on \(P_\Lambda\mathcal{H}\), together with the truncated Dirac operator
\(D_{\ell,\Lambda}=D_\ell|_{P_\Lambda\mathcal{H}}\).
The spectral truncations considered in this paper
are closely related to Peter--Weyl truncations.
Indeed, since
\[
D_\ell\Lambda(u_{ij}^{\alpha})
=\ell(\alpha)\Lambda(u_{ij}^{\alpha}),
\]
the spectral projection $P_\Lambda$ corresponds to
the finite set
\[
F_\Lambda=\{\alpha\in\operatorname{Irr}(G):
\ell(\alpha)\leq \Lambda\}.
\]

Our main contribution is to show that
higher-order commutators yield compact
quantum metric structures under suitable
polynomial growth conditions.
We then verify the conditions required
by the convergence framework of
Leimbach~\cite{6} and apply it
to these spectral truncations.
This provides a class of concrete examples
of quantum Gromov--Hausdorff convergence. The main result is that
\[
  \operatorname{dist}_q\!\big((C(G),L_\ell^k),
  (C(G)_\Lambda,L_{\ell,\Lambda}^{k,\Delta_\Lambda})\big)\longrightarrow 0
  \qquad \text{as } \Lambda\to\infty,
\]
whenever \(\big(C(G),L_\ell^k\big)\) is a compact quantum metric space.

Let $G$ be a compact quantum group with Hopf algebra of polynomial 
functions $\operatorname{Pol}(G)$. Let $\widehat{G}$ be the 
set of equivalence classes of irreducible unitary representations.
Let $\{u_{ij}^\alpha\}$ be the matrix coefficients of 
$\alpha \in \widehat{G}$, and let $(\mathcal{H}, \Lambda)$ denote the 
GNS representation of $C(G)$ associated with the Haar state. Then 
$\{\Lambda(u_{ij}^\alpha)\}$ forms an orthonormal basis of $\mathcal{H}$.
According to Definition~\ref{def:length}, for a proper length function, an 
unbounded Dirac operator 
$D_\ell:\Lambda(\operatorname{Pol}(G))\to \mathcal{H}$ can be defined by
\[
D_\ell\bigl(\Lambda(u_{i,j}^\alpha)\bigr) = \ell(\alpha)\,\Lambda(u_{i,j}^\alpha) .
\]
By \cite[Lemma~5.4]{23}, we know that 
$(\operatorname{Pol}(G),\mathcal{H},D_\ell)$ is a spectral triple. 

In \cite{4}, spectral truncation for the spectral triple actually is given by a spectral 
projection $P_\Lambda$ of $D_\ell$ onto the eigenspaces with eigenvalues of 
modulus $\leq\Lambda$. The $C^*$-algebra $C(G)$ is replaced by 
$P_\Lambda C(G) P_\Lambda$, acting on the Hilbert space 
$P_\Lambda \mathcal{H}$, and we set 
$D_{\ell,\Lambda}:=D_\ell|_{P_\Lambda\mathcal{H}}$. 
If $L_{C(G)}$ is the Lipschitz norm on $C(G)$ associated with $D_\ell$, we 
consider the quantum Gromov--Hausdorff distance between $C(G)$ and 
its truncation. If this distance tends to zero as $\Lambda\to\infty$, we say 
that the spectral truncation of $C(G)$ converges in quantum 
Gromov--Hausdorff distance. 

We recall the definition given by Li in~\cite{8}.
\begin{definition}(\cite{8} Definition 8.8)
    A Lipschitz seminorm  $L$ on a compact quantum group $(\mathcal{A}, \Delta)$ is 
    said to be \emph{left-invariant} (or \emph{right-invariant}) if for all 
    $a \in \mathcal{A}$ with $a^* = a$ and all $\mu \in S(\mathcal{A})$ we have
    \[
    L(a * \mu) \leq L(a) \quad (\text{or } L(\mu * a) \leq L(a)),
    \]
    where $a * \mu = (\mu \otimes \mathrm{id}) \Delta(a)$ and similarly for 
    $\mu * a$. Then $L$ is said to be \emph{bi-invariant} if it is both left 
    and right invariant.
\end{definition}

We now show that, for any $k\in \mathbb{N}$, $L_\ell^k$ is left-invariant. 
First, the comultiplication extends to a normal $*$-homomorphism 
$\Delta \colon L^{\infty}(G) \to 
L^{\infty}(G) \bar{\otimes} L^{\infty}(G)$, turning 
$L^{\infty}(G)$ into a compact von Neumann algebraic quantum 
group. On $L^2(G) \hat{\otimes} L^2(G)$ we have the 
left fundamental unitary $W$ and right fundamental unitary $V$ defined on the dense subspace 
$\Lambda \otimes \Lambda (C(G) \odot C(G))$ by the relation
\[
W^*(\Lambda(x) \otimes \Lambda(y)) = \Lambda \otimes \Lambda\bigl(\Delta(y)(x \otimes 1)\bigr),
\]
\[V(\Lambda(x) \otimes \Lambda(y)) = \Lambda \otimes \Lambda\bigl(\Delta(x)(1\otimes y)\bigr).\]
It implements the comultiplication on $C(G)$ by means of
\[
\Delta(x) = W^*(1 \otimes x)W=V(x\otimes1)V^*. \tag{2.3}
\]
The proof follows the idea of Lemma~4.5 in~\cite{16} and extends it to the $k$-th order setting.

\begin{lemma}\label{lemma left}
The seminorm \(L_\ell^k\) is bi-invariant, that is, for all
\(a \in \operatorname{Pol}(G)\) and
\(\varphi \in S(C(G))\),
\[
L_\ell^k\bigl((\varphi \otimes \operatorname{id})\Delta(a)\bigr)
\le \|\varphi\| \, L_\ell^k(a),
\]
and
\[
L_\ell^k\bigl((\operatorname{id}\otimes\varphi)\Delta(a)\bigr)
\le \|\varphi\| \, L_\ell^k(a).
\]
\end{lemma}

\begin{proof}
By Lemma~4.3 of \cite{16}, the operator
\(\operatorname{id}\otimes D_\ell\) commutes with \(W\) and \(W^*\).

For every \(a \in \operatorname{Pol}(G)\), we compute as follows:
all equalities are understood on the dense subspace
\(\Lambda(\operatorname{Pol}(G)) \odot
\Lambda(\operatorname{Pol}(G))
\subset L^2(G) \hat\otimes L^2(G)\):
\[
\begin{aligned}
&[\operatorname{id} \otimes D_\ell,\,
 [\operatorname{id} \otimes D_\ell,\,
  \dots,\,
  [\operatorname{id} \otimes D_\ell, \Delta(a)]\dots]] \\
&\quad = [\operatorname{id} \otimes D_\ell,\,
 [\operatorname{id} \otimes D_\ell,\,
  \dots,\,
  [\operatorname{id} \otimes D_\ell, W^*(1\otimes a)W]\dots]] \\
&\quad = W^*\,[\operatorname{id} \otimes D_\ell,\,
 [\operatorname{id} \otimes D_\ell,\,
  \dots,\,
  [\operatorname{id} \otimes D_\ell, 1\otimes a]\dots]]\,W.
\end{aligned}
\]

For every \(\varphi \in S(C(G))\), we obtain the following
identities for operators defined on the dense subspace
\(\operatorname{Pol}(G) \subset L^2(G)\):
\begin{align*}
L_\ell^k\bigl((\varphi \otimes \operatorname{id})\Delta(a)\bigr)
&= \bigl\| [D_\ell, [D_\ell, \dots, [D_\ell,
   (\varphi \otimes \operatorname{id})\Delta(a)]\dots]] \bigr\| \\
&= \bigl\| \varphi(a_{(0)}) \,
   [D_\ell, [D_\ell, \dots, [D_\ell, a_{(1)}]\dots]] \bigr\| \\
&= \bigl\| (\varphi \otimes \operatorname{id})
   [\operatorname{id} \otimes D_\ell,\,
    [\operatorname{id} \otimes D_\ell,\,
     \dots,\,
    [\operatorname{id} \otimes D_\ell, \Delta(a)]\dots]] \bigr\| \\
&= \bigl\| (\varphi \otimes \operatorname{id})
   \bigl( W^* (1 \otimes [D_\ell, [D_\ell, \dots, [D_\ell, a]\dots]]) W
   \bigr) \bigr\| \\
&\le \|\varphi \otimes \operatorname{id}\| \,
   \bigl\| W^* (1 \otimes [D_\ell, [D_\ell, \dots, [D_\ell, a]\dots]]) W
   \bigr\| \\
&= \|\varphi\| \, L_\ell^k(a).
\end{align*}
The same argument applies to right invariance, yielding
\[
L_\ell^k\bigl((\operatorname{id}\otimes\varphi)\Delta(a)\bigr)
\le \|\varphi\| \, L_\ell^k(a).
\]
This proves the claim.
\end{proof}

In~\cite{29}, De Ro and Hataishi extended this coaction definition to the setting of operator systems. We now recall the definition.
\begin{definition}(\cite{29} Definition 3.1)
A \emph{right coaction} $\alpha$ of the function algebra $A$ on the operator system $X$ is a uci map $\alpha: X \to X \otimes A$ such that the \emph{coaction property}
\begin{equation}\label{eq:coaction-prop}
(\alpha \otimes \mathbf{I}^A)\alpha = (\mathbf{I}^X \otimes \Delta)\alpha
\end{equation}
and the Podle\'s density condition
\[
\operatorname{span}((\mathbf{1}_X \otimes A)\alpha(X)) = X \otimes A
\]
are satisfied. A \emph{left coaction} $\beta: X \to A \otimes X$ is defined analogously.
\end{definition}

For a coaction \(\alpha\colon X\to X\otimes A\), an element \(x\in X\) is called a 
\emph{fixed point} if it satisfies \(x_{(0)}\otimes x_{(1)} = x\otimes \mathbf{1}_A\), 
and the set of all fixed points is denoted by 
\(X^\alpha := \{x\in X \mid \alpha(x)=x\otimes \mathbf{1}_A\}\). 
The coaction \(\alpha\) is said to be \emph{ergodic} if its only fixed points are 
scalar multiples of the unit, i.e.\ \(X^\alpha = \mathbb{C}\mathbf{1}_X\). 
For example, the comultiplication \(\Delta\colon A\to A\otimes A\) is ergodic.

We aim to construct the notion of a $C^k$-approximate order isomorphism, as given in Definition \ref{definition Ck}. In what follows, we verify that the compression map and symbol map introduced in \cite{6} indeed satisfy the defining conditions for a $C^k$-approximate order isomorphism. We first recall these two maps.

\begin{definition}(\cite{6} Definition 7.1)
Let $\tau_\Lambda\colon \mathcal{B}(H)\to \mathcal{B}(H_\Lambda)$ denote the \emph{compression map}, defined by
\[
\tau_\Lambda(T) := P_{\Lambda} T P_{\Lambda}
\]
for every $T\in \mathcal{B}(H)$. We write $A_\Lambda := \tau_\Lambda(A)\subseteq \mathcal{B}(H_\Lambda)$ for the image of the function algebra $A$ under the compression map.
\end{definition}

\begin{remark}
  Since $\tau_\Lambda:A\to A_\Lambda$ is surjective, the pullback map
  \[
    \tau_\Lambda^*:S(A_\Lambda)\to S(A),\qquad
    \tau_\Lambda^*(\phi_\Lambda)=\phi_\Lambda\circ\tau_\Lambda,
  \]
where $\phi_\Lambda\in S(A_\Lambda)$  is well defined.
\end{remark}

\begin{definition}(\cite{6} Definition 7.8)
Let $\phi \in \mathcal{S}(A_{\Lambda})$ be any state. We denote the associated slice map by $\rho_\Lambda^\phi : A_{\Lambda} \to A$, i.e.,
\[
\rho_\Lambda^\phi(x) := \phi(x_{(0)}) x_{(1)}
\]
for all $x \in A_{\Lambda}$. We call the map $\rho_\Lambda^\phi$ a \emph{symbol map}.
\end{definition}

First, we compute the compositions of the compression and symbol maps:
\begin{equation}
\rho_\Lambda^\phi \tau_\Lambda(a) = \phi(\tau_\Lambda(a)_{(0)}) \tau_\Lambda(a)_{(1)} = \tau_\Lambda^*\phi(a_{(0)}) a_{(1)},
\end{equation}
for all $a \in A$, and
\begin{equation}
\tau_\Lambda \rho_\Lambda^\phi(x) = \phi(x_{(0)}) \tau_\Lambda(x_{(1)}) = \tau_\Lambda^*\phi(a_{(0)}) \tau_\Lambda(a_{(1)}),
\end{equation}
for all $x \in A_{\Lambda}$ and $a \in A$ with $\tau_\Lambda(a) = x$ .

\begin{remark}
Note that $A_\Lambda$ is an operator system and
the compression map
$\tau_\Lambda:A\to A_\Lambda$
is a surjective ucp map.
By Proposition~7.4 of~\cite{6}, there exist
unique commuting ergodic right and left coactions
\[
\alpha^{\tau_\Lambda}:A_\Lambda
\longrightarrow A_\Lambda\otimes A,
\qquad
\beta^{\tau_\Lambda}:A_\Lambda
\longrightarrow A\otimes A_\Lambda,
\]
such that
\[
(\tau_\Lambda\otimes\mathrm{id}_A)\Delta
=
\alpha^{\tau_\Lambda}\tau_\Lambda,
\qquad
(\mathrm{id}_A\otimes\tau_\Lambda)\Delta
=
\beta^{\tau_\Lambda}\tau_\Lambda.
\]

For simplicity, write
$\alpha_\Lambda=\alpha^{\tau_\Lambda}$
and $\beta_\Lambda=\beta^{\tau_\Lambda}$.

The ergodicity of the induced coactions also
follows directly from the Haar invariance.
Indeed, if $y=\tau_\Lambda(a)$ satisfies
$\alpha_\Lambda(y)=y\otimes1_A$, then
\[
\begin{aligned}
y
&=(\mathrm{id}_{A_\Lambda}\otimes h)
  \alpha_\Lambda(y)\\
&=\tau_\Lambda
  \bigl((\mathrm{id}_A\otimes h)\Delta(a)\bigr)\\
&=h(a)I^{A_\Lambda}.
\end{aligned}
\]
Thus $\alpha_\Lambda$ is ergodic.
The same argument applies to $\beta_\Lambda$.

We now introduce the convolution notation.
For $x\in A_\Lambda$ and $\mu\in S(A)$,
define
\[
\mu*x
:=
(\mu\otimes\mathrm{id}_{A_\Lambda})
\beta_\Lambda(x),
\qquad
x*\mu
:=
(\mathrm{id}_{A_\Lambda}\otimes\mu)
\alpha_\Lambda(x).
\]
Both elements belong to $A_\Lambda$.

For $\phi\in S(A_\Lambda)$, we also write
\[
x*\phi
:=
(\phi\otimes\mathrm{id}_A)
\alpha_\Lambda(x),
\qquad
\phi*x
:=
(\mathrm{id}_A\otimes\phi)
\beta_\Lambda(x).
\]
These elements belong to $A$.

The coaction identities and the commutativity
of $\alpha_\Lambda$ and $\beta_\Lambda$
give the corresponding associativity relations.

Following Proposition~8.9 of Li~\cite{8}
and Proposition~6.4 of Leimbach~\cite{6},
for $x\in A_\Lambda$, define
\begin{equation}
\begin{aligned}
L'_{A_\Lambda}(x)
&:=
\sup_{\phi\in S(A_\Lambda)}
L_\ell^k(x*\phi),\\
L''_{A_\Lambda}(x)
&:=
\sup_{\phi\in S(A_\Lambda)}
L_\ell^k(\phi*x),\\
L_{\ell,\Lambda}^{k,\Delta_\Lambda}(x)
&:=
\max\left\{
L'_{A_\Lambda}(x),
L''_{A_\Lambda}(x)
\right\}.
\end{aligned}
\end{equation}
\end{remark}

We say that a seminorm \(L_X : X \to [0, \infty]\) is (right) \textit{invariant} for the right coaction \(\alpha\) if
\[
L_X(x_{(0)}\mu(x_{(1)})) \leqslant L_X(x),
\]
for all elements \(x \in X\) and states \(\mu \in \mathcal{S}(A)\). (Left) \textit{invariance} for a left coaction is defined analogously.
\begin{lemma}
The seminorm
$L_{\ell,\Lambda}^{k,\Delta_\Lambda}$
is a bi-invariant Lip-norm on $A_\Lambda$.
\end{lemma}

\begin{proof}
By definition,
$L_{\ell,\Lambda}^{k,\Delta_\Lambda}$
is a seminorm.
Since the coactions are $*$-preserving
and $L_\ell^k$ is $*$-invariant,
\[
L_{\ell,\Lambda}^{k,\Delta_\Lambda}(x^*)
=
L_{\ell,\Lambda}^{k,\Delta_\Lambda}(x).
\]
Moreover, the unitality of the coactions
implies
\[
\mathbb{C}I^{A_\Lambda}
\subseteq
\ker L_{\ell,\Lambda}^{k,\Delta_\Lambda}.
\]

Conversely, suppose that
$L_{\ell,\Lambda}^{k,\Delta_\Lambda}(x)=0$.
Since $\ker L_\ell^k=\mathbb{C}1_A$,
we have
\[
(\phi\otimes\mathrm{id}_A)
\alpha_\Lambda(x)
\in\mathbb{C}1_A
\]
for every $\phi\in S(A_\Lambda)$.

Since states separate points of $A_\Lambda$,
it follows that
\[
\alpha_\Lambda(x)=z\otimes1_A
\]
for some $z\in A_\Lambda$.
Applying $\mathrm{id}_{A_\Lambda}\otimes h$
and using Haar invariance, we obtain
\[
z
=
(\mathrm{id}_{A_\Lambda}\otimes h)
\alpha_\Lambda(x)
\in\mathbb{C}I^{A_\Lambda}.
\]
Consequently,
\[
\alpha_\Lambda(x)
\in\mathbb{C}I^{A_\Lambda}\otimes1_A.
\]
Applying the counit identity
\[
(\mathrm{id}_{A_\Lambda}\otimes\varepsilon)
\alpha_\Lambda
=
\mathrm{id}_{A_\Lambda},
\]
we conclude that
$x\in\mathbb{C}I^{A_\Lambda}$.
Hence
\[
\ker L_{\ell,\Lambda}^{k,\Delta_\Lambda}
=
\mathbb{C}I^{A_\Lambda}.
\]

We next establish bi-invariance.
Let $x=x^*\in A_\Lambda$ and
$\mu\in S(A)$.
By the bi-invariance of $L_\ell^k$
and the commutativity of the coactions,
\[
\begin{aligned}
L'_{A_\Lambda}(\mu*x)
&=
\sup_{\phi\in S(A_\Lambda)}
L_\ell^k((\mu*x)*\phi)\\
&=
\sup_{\phi\in S(A_\Lambda)}
L_\ell^k(\mu*(x*\phi))\\
&\leq
\sup_{\phi\in S(A_\Lambda)}
L_\ell^k(x*\phi)\\
&=
L'_{A_\Lambda}(x).
\end{aligned}
\]

Similarly,
\[
\begin{aligned}
L'_{A_\Lambda}(x*\mu)
&=
\sup_{\phi\in S(A_\Lambda)}
L_\ell^k((x*\mu)*\phi)\\
&=
\sup_{\phi\in S(A_\Lambda)}
L_\ell^k(x*(\mu*\phi))\\
&\leq L'_{A_\Lambda}(x),
\end{aligned}
\]
where $\mu*\phi$ denotes the state
on $A_\Lambda$ defined by
\[
(\mu*\phi)(y):=\phi(y*\mu),
\qquad y\in A_\Lambda.
\]

The same argument applies to
$L''_{A_\Lambda}$.
Therefore,
$L_{\ell,\Lambda}^{k,\Delta_\Lambda}$
is bi-invariant.

Finally, since $A_\Lambda$ is
finite-dimensional and
\[
\ker L_{\ell,\Lambda}^{k,\Delta_\Lambda}
=
\mathbb{C}I^{A_\Lambda},
\]
the seminorm
$L_{\ell,\Lambda}^{k,\Delta_\Lambda}$
induces the weak$^*$ topology on
$S(A_\Lambda)$.
Hence it is a bi-invariant Lip-norm.
\end{proof}

The Peter--Weyl truncations introduced in~\cite{6} correspond to the spectral truncations considered here, which are induced by the length-function Dirac operator acting on the corepresentation spaces. We therefore study the problem of Gromov--Hausdorff convergence within the framework established in~\cite{6}.

By Proposition~8.9 in Li~\cite{8}, if $(G,\Delta)$ is coamenable, then any regular Lip‑norm yields a bi‑invariant regular Lip‑norm. In the present paper, however, we prove that $L_\ell^k$ is already bi‑invariant, so the coamenability assumption on $(G,\Delta)$ is not required. 

As in~\cite{6}, we consider the relationship between $ L_{\ell,\Lambda}^{k,\Delta_\Lambda}$ and $ L_\ell^k$.
\begin{lemma}
    The maps $\tau_\Lambda: C(G) \to C(G)_\Lambda$ and $\rho_\Lambda^\phi : C(G)_\Lambda \to C(G)$ are unital and positive. Furthermore, for all $a \in C(G)$ and $x \in C(G)_\Lambda$, we have
    \[
        L_{\ell,\Lambda}^{k,\Delta_\Lambda}(\tau_\Lambda(a)) \leq L_\ell^k(a)
        \quad\text{and}\quad
        L_\ell^k(\rho^\phi(x)) \leq L_{\ell,\Lambda}^{k,\Delta_\Lambda}(x).
    \]
\end{lemma}

\begin{proof}
    The map $\tau_\Lambda$ is clearly unital and positive. Since  $\rho^\phi(x) =(\phi \otimes \mathbf{I}^A) \alpha^{\tau_\Lambda}(x)$ ,$\rho^\phi$ is the composition of the unital completely positive maps $\alpha^{\tau_\Lambda}$ and $\phi \otimes \mathbf{I}^A$; hence it is also unital and positive.
    
    For $a \in C(G)$ we compute:
    \begin{align*}
        L_{C(G)_\Lambda}^{'}(\tau_\Lambda(a))
        &= \sup_{\phi \in \mathcal{S}(C(G)_\Lambda)} L_{\ell}^k\bigl(\phi(\tau_\Lambda(a)_{(0)}) \tau_\Lambda(a)_{(1)}\bigr) \\
        &= \sup_{\phi \in \mathcal{S}(C(G)_\Lambda)} L_{\ell}^k\bigl(\tau_\Lambda^* \phi(a_{(0)}) a_{(1)}\bigr) \\
        &= \sup_{\psi \in \tau_\Lambda^*\mathcal{S}(C(G)_\Lambda)} L_{\ell}^k\bigl(\psi(a_{(0)}) a_{(1)}\bigr) \\
        &\leqslant \sup_{\psi \in \mathcal{S}(C(G))} L_{\ell}^k\bigl(\psi(a_{(0)}) a_{(1)}\bigr)\\
        &\leqslant L_{\ell}^k(a),
    \end{align*}
    where the inequality follows from the inclusion $\tau_\Lambda^*\mathcal{S}(C(G)_\Lambda) \subseteq \mathcal{S}(C(G))$, and the last equality uses the left invariance of $L_\ell^k$. The inequality for $\beta^\tau_\Lambda$ is analogous. Consequently,
    \[
        L_{\ell,\Lambda}^{k,\Delta_\Lambda}(\tau_\Lambda(a)) \leq L_\ell^k(a).
    \]
    
    Conversely, for $x \in C(G)_\Lambda$, we have
    \begin{align*}
        L_\ell^k(\rho^\phi(x))
        &= L_\ell^k(\phi(x_{(0)}) x_{(1)}) \\
        &\leq \sup_{\phi \in \mathcal{S}(C(G)_\Lambda)} L_\ell^k(\phi(x_{(0)}) x_{(1)}) \\
        &= L_{C(G)_\Lambda}^{'}(x).
    \end{align*}
    This completes the proof.
\end{proof}

By Proposition~6.8 of~\cite{6}, for coamenable compact quantum group $G$ we know that if \( L_\ell^k \) is a Lipschitz norm on \( C(G) \) with \( \ker(L_\ell^k) = \mathbb{C}1_{ C(G)} \), and $L_{C(G)_\Lambda}^{'}$  can be  the induced seminorm on $C(G)$. Let \( \mu, \nu \in S(C(G)) \) be states on \( A \) and consider the induced slice maps \( C(G)_\Lambda \to C(G)_\Lambda \), given by \( x \mapsto x_{(0)}\mu(x_{(1)}) \) and \( x \mapsto x_{(0)}\nu(x_{(1)}) \) respectively. Then the following holds, for all \( x \in C(G)_\Lambda \):
\[
\| x_{(0)}\mu(x_{(1)}) - x_{(0)}\nu(x_{(1)}) \| \leq 2 \, d^{L_\ell^k}(\mu, \nu) \, L_{C(G)_\Lambda}^{'}(x).
\]

Similarly, if \( \beta: {C(G)_\Lambda} \to C(G) \otimes {C(G)_\Lambda} \) is a left coaction and \( L_{C(G)_\Lambda}^{''} \) the induced seminorm on \( {C(G)_\Lambda} \), the following holds, for all \( x \in {C(G)_\Lambda} \):
\[
\| \mu(x_{(-1)}) \, x_{(0)} - \nu(x_{(-1)}) \, x_{(0)} \| \leq 2 \, d^{L_\ell^k}(\mu, \nu) \, L_{C(G)_\Lambda}^{''}(x).
\]
 If we consider $a\in C(G)$,since $L_\ell^k$ is bi-invariant,we obtain the following inclusion:
 \[
\| a_{(0)}\mu(a_{(1)}) - a_{(0)}\nu(a_{(1)}) \| \leq 2 \, d^{L_\ell^k}(\mu, \nu) \, L_\ell^k(a),
\]
and
\[
\| \mu(a_{(-1)}) \, a_{(0)} - \nu(a_{(-1)}) \, a_{(0)} \| \leq 2 \, d^{L_\ell^k}(\mu, \nu) \, L_\ell^k(a).
\]

Leimbach~\cite{6} studied the convergence of Peter--Weyl truncations for compact quantum groups and established a general framework for proving convergence in quantum Gromov--Hausdorff distance under suitable assumptions.  Rather than focusing solely on the general convergence framework for Peter--Weyl truncations, our work develops the metric properties needed to establish spectral truncation convergence for a range of compact quantum groups.

\begin{theorem}\label{theorem Gromov Hausdorff}
Let $G$ be a coamenable compact quantum group
equipped with a proper length function $\ell$.
Suppose that
$\bigl(C(G),L_\ell^k\bigr)$
is a compact quantum metric space. Then
\[
\operatorname{dist}_q\left(
  \bigl(C(G),L_\ell^k\bigr),
  \bigl(C(G)_\Lambda,
  L_{\ell,\Lambda}^{k,\Delta_\Lambda}\bigr)
\right)
\longrightarrow 0
\qquad\text{as }\Lambda\to\infty.
\]
\end{theorem}

\begin{proof}
Let $\tau_\Lambda$ denote the truncation map,
and let $\rho_\Lambda^\phi$ be the corresponding
reconstruction map.
For every $a\in C(G)$, we have
\begin{align*}
\bigl\|\rho_\Lambda^\phi\tau_\Lambda(a)-a\bigr\|
&=
\bigl\|
  \tau_\Lambda^*\phi(a_{(0)})a_{(1)}
  -\epsilon(a_{(0)})a_{(1)}
\bigr\|\\
&\leq
2d^{L_\ell^k}(\tau_\Lambda^*\phi,\epsilon)
L_\ell^k(a).
\end{align*}

Next, let $x\in C(G)_\Lambda$.
By Kadison's representation theorem and
Fubini's theorem for slice maps, we obtain
\begin{align*}
\bigl\|\tau_\Lambda\rho_\Lambda^\phi(x)-x\bigr\|
&=
\bigl\|
  \tau_\Lambda^*\phi(x_{(0)})\tau_\Lambda(x_{(1)})
  -\epsilon(x_{(0)})\tau_\Lambda(x_{(1)})
\bigr\|\\
&\leq
2\sup_{\psi\in S(C(G)_\Lambda)}
\Bigl|
  \psi\bigl(
    \tau_\Lambda^*\phi(x_{(0)})\tau_\Lambda(x_{(1)})
    -\epsilon(x_{(0)})\tau_\Lambda(x_{(1)})
  \bigr)
\Bigr|\\
&=
2\sup_{\psi\in S(C(G)_\Lambda)}
\Bigl|
  (\tau_\Lambda^*\phi-\epsilon)
  \bigl(x_{(0)}\tau_\Lambda^*\psi(x_{(1)})\bigr)
\Bigr|\\
&\leq
2d^{L_\ell^k}(\tau_\Lambda^*\phi,\epsilon)
\sup_{\psi\in S(C(G)_\Lambda)}
L_\ell^k\bigl(x_{(0)}\psi(x_{(1)})\bigr)\\
&=
2d^{L_\ell^k}(\tau_\Lambda^*\phi,\epsilon)
L_{\ell,\Lambda}^{k,\Delta_\Lambda}(x).
\end{align*}

By Corollary~7.12 of~\cite{6}, for every
$\varepsilon>0$ and all sufficiently large
$\Lambda$, the above estimates yield
\[
\bigl\|\rho_\Lambda^\phi\tau_\Lambda(a)-a\bigr\|
\leq \varepsilon L_\ell^k(a),
\qquad a\in C(G),
\]
and
\[
\bigl\|\tau_\Lambda\rho_\Lambda^\phi(x)-x\bigr\|
\leq \varepsilon
L_{\ell,\Lambda}^{k,\Delta_\Lambda}(x),
\qquad x\in C(G)_\Lambda.
\]

Consequently, $\tau_\Lambda$ and
$\rho_\Lambda^\phi$ form a
$C^k$-approximate order isomorphism
in the sense of
Definition~\ref{definition Ck}.
Applying Theorem~\ref{theorem:C^k},
we conclude that
\[
\operatorname{dist}_q\left(
  \bigl(C(G),L_\ell^k\bigr),
  \bigl(C(G)_\Lambda,
  L_{\ell,\Lambda}^{k,\Delta_\Lambda}\bigr)
\right)
\longrightarrow 0
\qquad\text{as }\Lambda\to\infty.
\]
This completes the proof.
\end{proof}

\begin{example}
By \cite[Corollary C]{16}, the pairs
\[
\bigl(C(\mathrm{SU}(2)),L_\ell\bigr),\quad
\bigl(C(O_2^+),L_\ell\bigr),\quad
\bigl(C(\mathrm{SO}(3)),L_\ell\bigr),\quad
\bigl(C(S_4^+),L_\ell\bigr)
\]
are compact quantum metric spaces with respect to
the length functions arising from their standard
fundamental corepresentations.

Consequently, by Theorem~\ref{theorem Gromov Hausdorff}
with $k=1$, for
\[
G=\mathrm{SU}(2),\ O_2^+,\ \mathrm{SO}(3),\ S_4^+,
\]
we obtain
\[
\operatorname{dist}_q\left(
\bigl(C(G),L_\ell\bigr),
\bigl(C(G)_\Lambda,
L_{\ell,\Lambda}^{1,\Delta_\Lambda}\bigr)
\right)
\longrightarrow 0
\qquad\text{as }\Lambda\to\infty.
\]
\end{example}

\begin{example}
Let \(G\) be a coamenable compact quantum group
of Kac type equipped with a proper length
function \(\ell\).
By Theorem~\ref{theorem fusion and C(G)},
if \((F(G),L_\ell^k)\) is a compact quantum
metric space, then so is \((C(G),L_\ell^k)\).
Consequently, by
Theorem~\ref{theorem Gromov Hausdorff},
we obtain
\[
\operatorname{dist}_q\left(
\bigl(C(G),L_\ell^k\bigr),
\bigl(C(G)_\Lambda,
L_{\ell,\Lambda}^{k,\Delta_\Lambda}\bigr)
\right)
\longrightarrow 0
\qquad \text{as }\Lambda\to\infty.
\]

In particular, by
Corollary~\ref{coro polynoimal forms cqms},
if \(G\) has polynomial growth of order \(s\),
then the above convergence holds for every
positive integer \(k>s\).

Moreover, the classical compact Lie groups
\(SU(N)\) and \(SO(N)\)
admit proper length functions with polynomial
growth. Therefore, the corresponding spectral
truncations converge in quantum
Gromov--Hausdorff distance for all sufficiently
large positive integers \(k\).
\end{example}

\bibliographystyle{amsplain}

\end{document}